\documentclass{amsart}
\usepackage{euler, hyperref, amssymb}
\usepackage[all]{xy}
\usepackage{color}
\usepackage[normalem]{ulem}

\theoremstyle{plain}
\newtheorem{theorem}{Theorem}[section]
\newtheorem{proposition}[theorem]{Proposition}
\newtheorem{lemma}[theorem]{Lemma}
\newtheorem{corollary}[theorem]{Corollary}

\theoremstyle{definition}
\newtheorem{definition}[theorem]{Definition}
\newtheorem{remark}[theorem]{Remark}

\newcommand{\f}{\varphi}

\newcommand{\CC}{\mathbb C}
\newcommand{\PP}{\mathbb P}

\newcommand{\bP}{\mathbf P}
\newcommand{\bR}{\mathbf R}
\newcommand{\bE}{\mathbf E}
\newcommand{\bX}{\mathbf X}

\newcommand{\C}{{\mathcal C}}
\newcommand{\E}{{ E}}
\newcommand{\F}{{ F}}
\newcommand{\G}{{ G}}
\newcommand{\I}{{\mathcal I}}
\def\O{\mathcal O}
\newcommand{\cU}{{\mathcal U}}
\newcommand{\T}{{\mathcal T}}

\newcommand{\Ker}{{\mathcal Ker}}
\newcommand{\Coker}{{\mathcal Coker}}
\newcommand{\Image}{{\mathcal Im}}
\newcommand{\Tor}{{\mathcal Tor}}

\newcommand{\Aut}{\operatorname{Aut}}
\newcommand{\Def}{\operatorname{Def}}
\newcommand{\Ext}{\operatorname{Ext}}
\newcommand{\ext}{\operatorname{ext}}
\newcommand{\Hom}{\operatorname{Hom}}
\newcommand{\h}{\operatorname{h}}
\def\H{\operatorname{H}}
\newcommand{\p}{\operatorname{p}}
\newcommand{\Hilb}{\operatorname{Hilb}}
\newcommand{\M}{\operatorname{M}}
\newcommand{\Ob}{\operatorname{Ob}}
\newcommand{\rank}{\operatorname{rank}}
\newcommand{\sing}{\operatorname{sing}}
\newcommand{\reg}{\operatorname{reg}}

\newcommand{\be}{\begin{equation}}
\newcommand{\ee}{\end{equation}}

\newcommand{\dual}{{\scriptscriptstyle \operatorname{D}}}
\newcommand{\trans}{{\scriptscriptstyle \operatorname{T}}}

\newcommand{\tensor}{\otimes}
\newcommand{\isom}{\simeq}
\newcommand{\lra}{\longrightarrow}

\newcommand{\ba}{\begin{array}}
\newcommand{\ea}{\end{array}}

\newcommand{\ses}[3]{0\rightarrow{#1}\rightarrow{#2}\rightarrow{#3}\rightarrow 0}
\newcommand{\lr}{\rightarrow}
\newcommand{\cF}{\mathcal{F}}
\newcommand{\cQ}{\mathcal{Q}}
\newcommand\cExt{\mathcal{E}xt }

\begin{document}

\begin{abstract}
As a continuation of the work of Freiermuth and Trautmann, we study the geometry of the moduli space
of stable sheaves on $\mathbb{P}^3$ with Hilbert polynomial $4m+1$.
The moduli space has three irreducible components whose generic elements are, respectively,
sheaves supported on rational quartic curves, on elliptic quartic curves, or on planar quartic curves.
The main idea of the proof is to relate the moduli space with the Hilbert scheme of curves by wall crossing.
We present all stable sheaves contained in the intersections of the three irreducible components.
We also classify stable sheaves by means of their free resolutions.
\end{abstract}

\title[Moduli of sheaves supported on quartic space curves]
{Moduli of sheaves supported on quartic space curves}

\author{Jinwon Choi}
\address{Department of Mathematics, Sookmyung Women's University, Seoul 140-742, Korea}
\email{jwchoi@sookmyung.ac.kr}

\author{Kiryong Chung}
\address{School of Mathematics, Korea Institute for Advanced Study, Seoul 130-722, Republic of Korea}
\email{krjung@kias.re.kr}

\author{Mario Maican}
\address{Center for Geometry and its Applications, Pohang University of Science and Technology,
Pohang 790-784, Republic of Korea}
\email{m-maican@wiu.edu}

\thanks{The second named author was partly supported by Korea NRF grant 2013R1A1A2006037.}

\maketitle

\section{Introduction}
\subsection{Motivations and results}
For a fixed polynomial $P(m)$ with rational coefficients, it is well-known that the space parameterizing Gieseker-stable sheaves on
a smooth projective variety $X$ with Hilbert polynomial $P(m)$ is a projective scheme (\cite{Sim94}).
In this paper, we are specifically interested in the case when $P(m)=dm+\chi$ is a linear polynomial,
as the moduli space is closely related to the relative Jacobian of families of curves.
When $X$ is the complex projective plane $\PP^2$,
Le Potier \cite{lepot} showed that the moduli space is an irreducible projective variety of dimension $d^2+1$,
that it  is locally factorial, and that its Picard group is generated by two explicitly constructed divisors.
When the leading coefficient $d$ is small, the moduli spaces have been studied by many authors.
Stratifications in terms of the free resolution of sheaves \cite{drezet_maican, maican_illinois, maican6},
topological invariants such as Poincar\'{e} polynomials \cite{cm, cc1, cc2, yuan1}, stable base locus decompositions \cite{cc2} are known.

In this paper we study the case when $X$ is $\PP^3$.
Let $\M(P(m))$ denote the moduli space of stable sheaves on $\PP^3$ with Hilbert polynomial $P(m)$.
Freiermuth and Trautmann \cite{freiermuth_trautmann} examined $\M(3m+1)$.
They gave a complete classification of sheaves in $\M(3m+1)$
and proved that this moduli space has two irreducible components intersecting transversally.
The first component which we denote by $R_3$ parametrizes the structure sheaves of twisted cubic curves.
The second component parametrizes the sheaves of the form $\O_C(p)$, where $C$ is a planar cubic curve and $p\in C$.

From the point of view of the birational geometry, $R_3$ is related to the moduli space of twisted cubic curves.
Let $\overline{M}_0(\PP^r, d)$ be the moduli space of genus zero stable maps to $\PP^r$ of degree $d$.
It is known that this is a projective normal variety containing a Zariski dense open subset consisting of irreducible rational curves of degree $d$.
Specially, if $r=3$ and $d=3$, the moduli space $\overline{M}_0(\PP^3,3)$ of stable maps
is the unique flipping space of the component $R_3$ over the (normalization of the) Chow variety (see \cite{Che08}).
We remark that the flipping locus is exactly the locus of multiple covers of its image or the stable sheaves supported on non-reduced curves.
More generally on $\PP^r$, Chen, Coskun, and Crissman \cite{CCC11} proved that the moduli space of stable maps $\overline{M}_0(\PP^r, 3)$
is one of the flip models of the compactified space of stable sheaves supported on  rational curves. % through the similar definition of the compactification of rational curves in the moduli space $\M_{\PP^r}(3m+1)$.
On the other hand, in \cite{CK11, CHK12}, by extending the flip map between the compactified moduli spaces of rational curves,
the authors established an explicit birational relation between $\overline{M}_0(X, 3)$ and $R_3(X)$ when $X$ is a projective homogenous variety,
which is a natural generalization of the projective space.
Until now, to the knowledge of the authors, there are no such results for higher degrees $d$.
Thus, our project can be regarded as a starting point of understanding the various compactified moduli spaces of curves in $\PP^r$.

In this paper we investigate $\M(4m+1)$.
\begin{definition}
\hfill
\begin{itemize}
\item The dual of a one-dimensional sheaf $F$ on $\PP^3$ is defined by $F^\dual = \cExt^2(F,\omega_{\PP^3})$.
\item We call a sheaf $F$ \emph{planar} if $F \isom F|_H$ for some plane $H\subset \PP^3$, or equivalently,
the schematic support of $F$ is contained in $H$.
%\item We denote the three components by $\bR$, $\bE$, $\bP$.
\end{itemize}
\end{definition}
Our main result is that the space $\M(4m+1)$ consists of three irreducible components.
\begin{theorem}\label{thm:main}
Let $\M(4m+1)$ be the moduli space of stable sheaves in $\PP^3$ with Hilbert polynomial $4m+1$.
Then $\M(4m+1)$ consists of three irreducible components whose general points are
\begin{enumerate}
\item $\bR:$ the structure sheaves of rational quartic curves;
\item $\bE:$ the dual sheaves $I_{p,C}^\dual$ of the ideal sheaves $I_{p,C}$ of points $p$ on elliptic quartic curves $C$;
\item $\bP:$ the planar sheaves.
\end{enumerate}
\end{theorem}

The component $\bR$ is the compactification of the space of rational quartic curves which is predicted as a birational flip model of the space of finite maps (cf. \cite{chen_nollet}).
The general member of $\bE$ is a line bundle of the form $\O_C(p)$, where $p\in C$, for a smooth elliptic quartic curve $C$.
Since $\Ext^1(\CC_p,\O_C)=\CC$, this line bundle fits into the unique non-split extension
\[
\ses{\O_C}{\O_C(p)}{\CC_p}.
\]
Using the results in \cite{vainsencher}, one can easily see that $\bE$ is birational to the universal elliptic quartic curve (For detail, see Lemma \ref{lem:e1}).
As the supports of sheaves in $\bE$ degenerate to singular quartic curves of arithmetic genus $1$, two complications arise.
Firstly, the dimension of $\Ext^1(\CC_p,\O_C)$ may jump to $2$ (Lemma \ref{jumpext}).
Secondly, some extension sheaves may fail to be stable because they may have a destabilizing subsheaf that is isomorphic to a stucture sheaf $\O_L$ of a line $L$,
which is given by the push-out of $\Ext^1(\CC_p,\O_L(-1))=\CC$, where we have an exact sequence
\[
\ses{\O_L(-1)}{\O_C}{\O_{C_0}}
\]
for some planar cubic curve $C_0$. We overcome these difficulties by regarding such a locus as a bundle space over other base spaces.

%\subsection{Idea of the proof of Theorem \ref{thm:main}}
\medskip
The main idea of the proof of Theorem \ref{thm:main} is to relate $\M(4m+1)$ with the Hilbert scheme of degree $4$ curves by using wall crossing, which we review in \S 2.
It is known that this Hilbert scheme consists of four irreducible components.
One component consists of curves corresponding to unstable sheaves and is irrelevant for our purposes.
The other three components are modified by wall crossing into $\bR$, $\bE$ and $\bP$.
 %of pairs studied by Le Potier, Min He, and Stoppa and Thomas. From the Hilbert scheme of curves with Hilbert polynomial $4m+1$, one can easily see the moduli points of the moduli space of the infinity stable pairs.  The key issue is that unlikely degree $3$, the extension class of $\O_C$ by $\CC_p$ may not be unique (up to scalar multiplication). We carefully study this locus and see how it changes during the wall-crossing. It turns out that this exceptional locus is a exceptional divisor of a blow-up space. Or it is an subspace of a projective bundle over irreducible variety. Finally, as we classifies the possible stable sheaves $\M(4m+1)$ we finish the proof.

In \S\ref{sec:int} we describe the intersections of $\bR$, $\bE$ and $\bP$ by using the {\em elementary modifications} of sheaves.
The sheaves in $\bR\cap \bP$ and $\bE\cap\bP$ can be classified by using the description of the incident variety of the planar quartic curves with points (concretely, the relative Hilbert scheme) (see \cite{cc1}).
The intersection $\bR\cap \bE$ is more complicated.
A natural candidate for the general element in the intersection is provided by a pair of a non-degenerate singular elliptic curve and its singular point.
The key issue is to prove that every degeneration of stable sheaves in $\bR \cap \bE$ is a limit of such pairs.
It seems hard to describe the boundary of $\bR$ in the Hilbert scheme, so,
instead, we use the modification technique developed in \cite{CK11, CHK12}.
The authors in \cite{CK11, CHK12} compare the moduli space of stable maps (or finite maps)
with the moduli space of stable sheaves when the degree of the maps (or sheaves) is at most $3$.
Starting from the canonical family of pure sheaves obtained as direct images of stable maps,
they extend the birational maps into the birational \emph{regular} maps
after performing modifications of sheaves several times along the exceptional divisors of blowing-ups of the stable maps space.
%This provide us some hint to see of which stable sheaves can arise as the boundary of the space of rational curves.
In our case of degree $4$, even though we do not have a full picture of the blowing-ups and modifications of sheaves,
we can find the normal direction of the space of stable maps which provides the indicated stable sheaves.
By combining with the computation of the deformation space, we show that $\bR \cap \bE$ is a single irreducible divisor in $\bR$.

For completeness, we present in the last section the possible free resolutions of the sheaves in $\M(4m+1)$.
These can be used in the local analysis of each component in a forthcoming work.

\medskip
\textbf{Acknowledgement.}
We would like to thank Sheldon Katz and Han-Bom Moon for valuable discussions and comments.

\section{Review of the wall crossing}
\subsection{General framework}\label{wallsection}

In this section we review the wall crossing technique that we will use in the paper.
Motivated by the Donaldson-Thomas/Pandharipande-Thomas correspondence \cite{PT09},
Stoppa and Thomas \cite{st} studied a GIT wall crossing between the Hilbert scheme of curves and the moduli space of stable pairs.
A stable pair here is a pair of a one-dimensional sheaf and a section which generates the sheaf away from finitely many points.
Both of the Hilbert scheme and the moduli space of stable pairs are equipped with a perfect obstruction theory and hence the virtual invariants are defined by integrating the cohomology cycle against the virtual fundamental classes (\cite{Thomas, PT09}).
Stoppa and Thomas realized both spaces as GIT quotients of a certain space of pairs
and by altering the linearization they showed that these two moduli spaces are related by GIT wall crossing.

One can go further. In \cite{cc1}, the authors study the wall-crossing for the moduli space $\M^\alpha(P(m))$ of $\alpha$-stable pairs,
where $\alpha$ is a nonnegative rational number.
\begin{definition}
\begin{enumerate}
\item A {\em pair} is a pair of a sheaf and a non-zero section.
\item A sheaf $F$ is \emph{pure} if for every nonzero subsheaf $G\subset F$ the dimension of the support of $G$ is the same as the dimension of the support of $F$.
\item A pair $(s,F)$ is called \emph{$\alpha$-semistable} if $F$ is pure and for any proper non-zero subsheaf $F'\subset  F$, the inequality
\[
\frac{\chi(F'(m))+\delta\cdot\alpha}{r(F')} \leq \frac{\chi(F(m))+\alpha}{r(F)}
\]
holds for $m\gg 0$. Here $r(F)$ is the leading coefficient of the Hilbert polynomial $\chi(F(m))$ and
$\delta=1$ if the section $s$ factors through $F'$ and $\delta=0$ otherwise.
When the strict inequality holds, $(s,F)$ is called \emph{$\alpha$-stable}. If a pair is $\alpha$-semistable but not $\alpha$-stable, it is called strictly $\alpha$-semistable.
\end{enumerate}
\end{definition}

A pair is a special case of the coherent system by Le Potier \cite{LePotier}. A coherent system is a pair of a sheaf with a subspace $V$ of $H^0(F)$. So, a pair is a coherent system when the dimension of $V$ is one. We will use the notation $(1,F)$ to denote the pair of the sheaf $F$ with its nonzero section and $(0,F)$ when the section is taken to be zero. Although by definition the section in a pair must be nonzero, these notation is convenient when we study the $\alpha$-stability. The short exact sequence
$$0 \to (1,G)\to (1,F)\to (0,F/G)\to 0 $$ indicates that $G$ is a subsheaf of $F$ and the section $s$ of $F$ is in fact in $H^0(G)$ so that $\delta=1$. Similarly the short exact sequence $$0 \to (0,G)\to (1,F)\to (1,F/G)\to 0 $$ now means that the section $s$ of $F$ does not factor through $G$ and hence $\delta=0$ and we take the image of the section $s$ in $F/G$.

Similarly as in the sheaf case, we can define a Jordan-H\"older filtration and S-equivalence classes. The S-equivalence classes of $\alpha$-semistable pairs form a projective moduli space $\M^\alpha(P(m))$  \cite{LePotier}. We say $\alpha\in \mathbb{Q}$ is a {\em wall} if there is a strictly $\alpha$-semistable pairs. As we change $\alpha$, the moduli space $\M^\alpha(P(m))$ changes only at walls \cite{Tha}. A simple reason for this is that if there is no strictly semistable pairs, $\alpha$-stability does not change by a small perturbation of $\alpha$ because the stability condition is defined by a strict inequality. After we fix the Hilbert polynomial $P(m)$, it is easy to see there are only finitely many walls. By the {\em wall crossing} at $\alpha_0$, we mean to compare the moduli spaces for $\alpha_-<\alpha_0<\alpha_+$ which are sufficiently close to $\alpha_0$, that is, there is no other walls between $\alpha_-$ and $\alpha_+$.

What happens at the wall $\alpha_0$ is as follows. Let $(1,G)\oplus (0,H)$ be a strictly $\alpha_0$-semistable pair. Then the pair $(1,F)$ obtained by the extension
\begin{equation}\label{eq:pst} 0 \to (0,H)\to (1,F)\to (1,G)\to 0 \end{equation}
is $\alpha_+$-stable but $\alpha_-$-unstable, while the pair $(1,F')$ obtained by the extension
\begin{equation}\label{eq:mst} 0 \to (1,G)\to (1,F')\to (0,H)\to 0 \end{equation}
is $\alpha_-$-stable but $\alpha_+$-unstable. So when crossing the wall, the pairs of the form \eqref{eq:pst} are replaced by the pairs of the form \eqref{eq:mst}. This can be explained by the {\em elementary modification} of the pair. See \cite{cc1, Tha} for more details.

By the definition of $\alpha$-stability, when $\alpha$ tends to infinity (for short $\alpha=\infty$),
the cokernel of the pair $s \colon \O \to F$ is supported on a zero-dimensional scheme (possibly empty).
In other words, we get the moduli space of Pandharipande-Thomas stable pairs.
On the other hand, when $\alpha$ is sufficiently small (for short $\alpha=0^+$),
$\alpha$-stability is equivalent to the Gieseker stability of the sheaf.
Thus, by wall crossing, conditions on the section are replaced by conditions on the sheaf.
Now, since there is no condition on sections,  we get a map to our moduli space $\M(P(m))$ of sheaves by forgetting the section.

When $P(m)=4m+1$, we see that there is only one wall at $\alpha=3$ (\cite{cc1}). Strictly $\alpha$-semistable pair in this case is $(0,\O_L)\oplus(1,G),$ where $L$ is a line and $G$ is a sheaf with Hilbert polynomial $3m$.
%We denote by $(1,F)$ the pair of a sheaf $F$ with a non-zero section and by $(0,F)$ the pair of sheaf $F$ with zero section.
The pair given by the extension
\[
\ses{(0,\O_L)}{(1,F)}{(1,G)},
\]
is unstable for $\alpha< 3$ because $\frac{1}{1}>\frac{1+\alpha}{4}$.
After crossing the wall, this pair is modified into a pair given by the ``flipped'' extension
\[
\ses{(1,G)}{(1,F)}{(0,\O_L)}.
\]

The wall crossing picture is as follows.
\[
{\small
\xymatrix{
\Hilb_{\PP^3}(4m+1) \ar[dr] \ar@{<--}[r]&{\text{flip}} & \M^{\infty}(4m+1) \ar[dl] \ar[dr] \ar@{<--}[l]\ar@{<--}[r]& \text{flip}& \M^{0^+}(4m+1) \ar[dr]^{\text{forgetful map}} \ar[dl] \ar@{<--}[l]& \\
& \operatorname{Chow} & & \M^{\alpha=3}(4m+1) & & \M(4m+1).
}
}
\]
Here the Chow variety is defined by
\[
\operatorname{Chow} = \coprod_{g=0,1,3} [CM^{4m+1-g}\times S^g(\PP^3)]
\]
where $CM^{4m+1-g}$ is the space of Cohen-Macaulay (shortly, CM) curves with Hilbert polynomial $4m+1-g$ and
$S^g(\PP^3)$ is the $g$-fold symmetric product. Note that there does not exist any CM curve of genus $g=2$ in $\PP^3$ (\cite[Theorem 3.3]{Har94}).
\begin{lemma}\label{pairsstable}
\hfill
\begin{enumerate}
\item If a pair $(s,F)$ is $\alpha$-semistable then the (scheme theoretic support) of $F$ is a CM-curve.
\item A pair $(s,F)$ is $\infty$-stable
if and only if the cokernel of $s$ is supported on a zero-dimensional subscheme of the support of $F$.
\item Assume that $\gcd(d,\chi)=1$. Then $(s,F)$ is $0^+$-stable if and only if $F$ is a stable sheaf.
\end{enumerate}
\end{lemma}
\begin{proof}
Part (1) is by the purity of $F$. Part (2) directly comes from \cite[\S 1.3]{PT09}. For part (3), when $\gcd(d,\chi)=1$, there is no strictly (Gieseker) semistable sheaves. When $\alpha=0$, the $\alpha$-stability is the same as the Gieseker stability by definition and a small perturbation of $\alpha$ does not change the stability.
\end{proof}

\begin{remark}\label{3mplusonewall}
\hfill
\begin{enumerate}
\item By a simple calculation, one can see that there are no walls for $d=1,2$, so $\M^{\alpha}(dm+1)$ does not change when $\alpha$ varies.
\item As a first non-trivial case, when the Hilbert polynomial is $P(m)=3m+1$,
Freiermuth and Trautmann proved that $\M(3m+1)$ consists of two smooth irreducible components $R \cup E$.
Here $R$ parametrizes the structure sheaves of the twisted cubic curves and $E$ is isomorphic to the universal cubic plane curve.
By analyzing the deformation space of each sheaf in $R \cap E$, they showed that $R$ and $E$ meet transversally.
This can be explained as follows by using wall crossing.
According to Piene and Schlessinger \cite{ps}, the Hilbert scheme $\Hilb_{\PP^3}(3m+1)$ consists of two irreducible components:
rational cubic curves, respectively, elliptic curves with a point.
After wall crossing, $\M^{\infty}(3m+1)$ consists of two irreducible components $R\cup E$ as above.
Note that the locus of elliptic curves with a free point is excluded after performing the wall-crossing.
There is no wall-crossing for pairs and all sheaves $F$ satisfy $h^0(F)=1$ so that the forgetful map is an isomorphism.
Thus, we have
\[
\M^{\infty}(3m+1) \isom \M^{0^+}(3m+1) \isom \M(3m+1).
\]
\end{enumerate}
\end{remark}

\subsection{Geometry of the Hilbert scheme}\label{sec:hilb}

The irreducible components of the Hilbert scheme $\Hilb_{\PP^3}(4m+1)$ have been described in \cite{chen_nollet}.
\begin{proposition}(\cite[Theorem 4.9]{chen_nollet})\label{prop:hilb41}
The Hilbert scheme of curves with Hilbert polynomial $4m+1$ in $\PP^3$ consists of four irreducible components:
\begin{enumerate}
\item The closure of the locus of the rational quartic curves.
\item The closure of the locus of the unions of a line and a planar cubic curve.
\item The closure of the elliptic quartic curves with one isolated point.
\item The closure of the planar quartic curves with three isolated points.
\end{enumerate}
\end{proposition}
%\begin{remark}
%The curves in the first two locus (1) and (2) are CM-curves and the others are not.
%\end{remark}

The space of elliptic quartic curves in (3) has been studied in \cite{vainsencher,chen_nollet}. We remark that every connected CM-curve of degree $d=4$ and genus $g=1$ is a ACM-curve and non-planar (\cite[Theorem 3.3 and Proposition 3.5]{Har94}).

\begin{proposition}\label{prop:hilb}
\hfill
\begin{itemize}
\item (\cite[\S 5]{vainsencher}) The closure of CM-curves component in $\Hilb_{\PP^3}(4m)$ is obtained from the Grassmannian $\operatorname{Gr}(2,10)$ by blowing up twice,
where the blow-up loci are degeneracy loci of determinantal varieties.
\item (\cite[Theorem 5.2]{vainsencher}) The possible types (up to projective equivalence) of defining ideals for the connected CM-curve $C$ are:
\begin{enumerate}
\item $I_C=\langle q_1,q_2\rangle$, where $q_1$ and $q_2$ are quadratic polynomials.
\item $I_C=\langle xy,xz,yq_1+zq_2\rangle$; here $C$ is the union of a planar cubic curve and a line meeting at a point.
\item $I_C=\langle x^2,xy,xq_1+yq_2\rangle $; here $C$ is the union of a planar conic curve and a double line with genus $-2$.
\end{enumerate}
\item (\cite[Example 2.8 (b)]{chen_nollet}) The ideal sheaf of the CM-curve $C$ with Hilbert polynomial $4m$ has the resolution
\[
\ses{\O(-4)\oplus \O(-3)}{\O(-3)\oplus \O(-2)^{\oplus 2}}{I_C}.
\]
\end{itemize}
\end{proposition}

%\subsection{Strategy of proof}
%To prove Theorem \ref{thm:main}, we will use the following strategy. First we show $\M^\infty(4m+1)$ consists of four irreducible components $\bR^\infty$, $\bX^\infty$, $\bE^\infty$, and $\bP^\infty$ corresponding each of four irreducible components in Proposition \ref{prop:hilb41}.

\section{The moduli space $\M^{\infty}(4m+1)$}
%\subsection{Components of $\M_{\PP^3}^{\infty}(4m+1)$}
In this section we will describe all stable pairs in $\M_{\PP^3}^{\infty}(4m+1)$.
Considering the Hilbert-Chow morphism and the Hilbert scheme described in \S \ref{sec:hilb}, we classify all stable pairs according to their supports.

\begin{lemma}\label{class}
Let $(s,F) \in \M_{\PP^3}^{\infty}(4m+1)$ be a stable pair. Then,
\begin{enumerate}
\item $s \colon \O_C \stackrel{\isom}{\to} F$ where $C$ is a rational quartic curve or the disjoint union of a line and a planar cubic curve.
Note that $\chi(\O_C(m)) = 4m + 1$.
\item $s \colon \O_C \hookrightarrow F$ such that ${\mathcal Coker}(s)=\CC_p$ for a point $p\in C$. Note that $\chi(\O_C(m))=4m$.
\item $s \colon \O_C \hookrightarrow F$ such that $K = {\mathcal Coker}(s)$ has dimension zero and length $3$. Note that $\chi(\O_C(m))=4m-2$.
\end{enumerate}
\end{lemma}

\begin{proof}
By stability, the cokernel of $s$ must be supported on a zero-dimensional subscheme of the support of the sheaf (Lemma \ref{pairsstable}).
By using the classification of CM-curves given at \S \ref{sec:hilb} and comparing Hilbert polynomials,
one can check that the above list is complete.
\end{proof}

\begin{remark}
In case (1) of Lemma \ref{class}, the set of rational quartic curves is nothing but the Hilbert scheme of connected CM-curves with Hilbert polynomial $4m+1$,
which is well-known to be irreducible (cf. \cite{NS03}). We denote this component by $\bR^{\infty}$.
We denote by $\bX^{\infty}$ the component consisting of disjoint unions of a line and a planar cubic curve.
This component will be dropped from the moduli space after wall crossing.

In case (3), the sheaf $F$ must be planar. Let $F$ be the pure sheaf fitting into the short exact sequence
\[
\ses{\O_C}{F}{K}
\]
with $l(K)=3$. By the genus-degree formula, the curve $C$ has degree $4$ and is contained in a plane $H\subset \PP^3$.
Let us consider the natural restriction map
\[
\phi \colon F \twoheadrightarrow F|_H.
\]
Since $F|_H$ has the subsheaf $\O_C$, ${\mathcal Ker}(\phi)$ is zero-dimensional or zero.
In the first case the purity of $F$ gets contradicted. Thus $F \isom F|_H$, that is, $F$ is planar.
%Firstly, by degree genus formular, $C$ should be a plane curve in $H$. Secondly, by definition the sheaf $F$ is pure one-dimensional sheaf and thus it is reflexive. By taking the Maican dual, we obtain an exact sequence
%\[
%0=Ext^2(K,\omega)\lr Ext^2(F,\omega):=F^\dual\lr \Ext^2(\O_C, \omega)\lr \Ext^3(K,\omega)
%\] The first term is zero since $\dim (K)=0$. Also, $\Ext^2(\O_C, \omega)\cong \O_C(1)$. Hence, $F^\dual$ is a planar sheaf. By taking the dual $F^{DD}\cong %F$ again, we obtain the planar sheaf $F$. Note that $(i_*G)^\dual=i_*(G^\dual)$ for $i:H\subset \PP^3$ and any coherent sheaf $G$ on $H$.
Such sheaves form an irreducible component, denoted $\bP^\infty$,
which is isomorphic to the relative moduli space $\M^{\infty}(\PP \cU,4m+1)$, where $\cU$ is the universal rank three bundle over $Gr(3,4)=(\PP^3)^*$.

In case (2) of Lemma \ref{class}, the sheaf $F$ is given by the extension
\[
\ses{\O_C}{F}{\CC_p},
\]
where $C$ is as in Proposition \ref{prop:hilb}. However, unlike in the degree 3 case studied by Freiermuth and Trautmann,
the sheaf $F$ may not be uniquely determined by a pair $(p,C)$,
that is, the extension group $\Ext^1(\CC_p,\O_C)$ can have dimension greater than one.
In fact, $\dim \Ext^1(\CC_p,\O_C) \leq 2$ (see Lemma \ref{jumpext}).
We denote the locus of such pairs by $\bE^{\infty}=\bE_1^{\infty}\cup \bE_2^{\infty}$, where the lower index is $\dim \Ext^1(\CC_p,\O_C)$.
\end{remark}

\begin{lemma}\label{jumpext}
Let $C$ be a CM-curve with Hilbert polynomial $4m$ and let $p \in C$. Then
\[
1\leq \dim \Ext^1(\CC_p,\O_C)\leq 2.
\]
Moreover,
$
\Ext^1(\CC_p,\O_C) \isom \CC^2
$
precisely when (up to projective equivalence)
\begin{enumerate}
\item $I_{C}=\langle xy,xz,yq_1+zq_2\rangle$ and $q_1(p)=q_2(p)=0$,
%where $I_L=<y,z>$, $I_{C_0}=<x, yq_1+zq_2>$, and $p\in L\cap \sing(C_0)$.
\item $I_C=\langle x^2,xy,xq_1+yq_2\rangle$ and $q_1(p)=q_2(p)=0$.
\end{enumerate}
\end{lemma}

\begin{proof}
By Serre duality (\cite[Theorem 3.12]{Huy06}), and using the short exact sequence
\[
\ses{I_C}{\O_{\PP^3}}{\O_C},
\]
we have the isomorphisms
\[
\Ext^1(\CC_p,\O_C) \isom \Ext^2 (\O_C,\CC_p)^* \isom \Ext^1 (I_C,\CC_p)^*
\]
From the resolution of $I_C$ at Proposition \ref{prop:hilb}, we get the exact sequence
\begin{multline}
\label{sequenceofP1}
0 \to \Ext^0 (I_C,\CC_p) \to \Ext^0(\O(-3)\oplus2\O(-2),\CC_p) \xrightarrow{\delta} \Ext^0(\O(-4)\oplus \O(-3),\CC_p) \to \\
\Ext^1 (I_C,\CC_p) \to 0.
\end{multline}
Now we calculate the rank of $\delta$ for each case described in Proposition \ref{prop:hilb}. Without loss of generality, we may assume that $p=[0:0:0:1]$.
\begin{enumerate}
\item If $C$ is general, meaning $I_C=\langle q_1,q_2\rangle $, then $\delta$ is given by the matrix
\[
\left[\begin{array}{cc}0 & -q_2 \\0 & q_1 \\ 1 & 0 \\\end{array}\right]
\]
which has rank $1$ at $p$. Thus, $\Ext^1 (\CC_p,\O_C) \isom \CC$.
\item Assume that $C$ is the union of the line $L$ and cubic curve $C_0$.
Write
\[
I_C=\langle xy,xz,yq_1+zq_2\rangle, \qquad I_L=\langle y,z\rangle, \qquad I_{C_0}=\langle x, yq_1+zq_2\rangle.
\]
Then $\delta$ is given by the matrix
\[
\left[\begin{array}{cc}-q_1 & z \\-q_2 & -y \\ x & 0 \\\end{array}\right]
\]
whose rank at $p$ depends on the position of $p$ as follows:
\begin{enumerate}
\item $\rank (\delta(p)) = 0$ if and only if $p \in \sing(C_0)$ (i.e., $q_1(p)=q_2(p)=0$) and  $p \in L$.
In this case, $\Ext^1(\CC_p,\O_C) \isom \CC^2$;
\item $\rank(\delta(p)) = 1$, otherwise. In this case, $\Ext^1(\CC_p,\O_C) \isom \CC$.
\end{enumerate}
\item Assume that $C$ is the union of a conic curve and a double line of genus $-2$. Write
\[
I_C=\langle x^2,xy,xq_1+yq_2\rangle, \qquad I_L=\langle x,y\rangle, \qquad I_Q=\langle x,q_2\rangle.
\]
Then $\delta$ is given by the matrix
\[
\left[\begin{array}{cc}-q_1 & -y \\-q_2 &x \\ x & 0 \\\end{array}\right].
\]
As before, we see that $\Ext^1(\CC_p,\O_C) \isom \CC^2$ if and only if $q_1(p)=q_2(p)=0$.
This happens when $p \in L\cap Q$ and $q_1(p)=0$. \qedhere
\end{enumerate}
\end{proof}
%\begin{definition}
%Let $p\in C$ be the closed point. Let us define the embedding dimension of the curve $C$ at $p$ by $\dim(T_pC)$. From the above lemma, we know that $\Ext^1(\CC_p,\O_C)\cong \CC^2$ if and only if $\dim T_pC=\dimT_p\PP^3=3$.
%\end{definition}

%Using Lemma \ref{jumpext}, one can classify all stable pairs along its supporting curves. For later use, let us address the result.

\noindent
We denote by $\langle C\rangle $ the maximal linear space containing the curve $C$. From Lemma \ref{jumpext}, we obtain:

\begin{proposition}
\label{prop:einf}
The moduli space $\M^{\infty}(4m+1)$ is the union
$
\bR^{\infty}\cup \bE^{\infty}\cup\bP^{\infty}\cup \bX^{\infty}.
$
Furthermore, $\bE^{\infty} = \bE_1^{\infty} \cup \bE_2^{\infty}$, where
\begin{enumerate}
\item $\bE_1^{\infty}$ is the set of non-split extensions of $\CC_p$ by $\O_C$ such that one of the following holds:
\begin{enumerate}
\item $I_C=\langle q_1,q_2\rangle$;
\item $I_C=I_{L\cup C_0}$, where $C_0$ is a planar cubic curve, $L$ is a line meeting $C_0$ at a point, and $p$ is a point in $C_0$
such that either $p\notin \sing(C_0)$ or $p\notin L$;
\item $I_C=\langle x^2,xy,xq_1+yq_2\rangle$ after a change of coordinates,
and $p$ is a point on $C$ such that $p \notin L$ or $q_1(p) \neq 0$ or $q_2(p) \neq 0$;
\end{enumerate}
\item $\bE_2^{\infty}$ is the set of non-split extensions of $\CC_p$ by $\O_C$ such that one of the following holds:
\begin{enumerate}
\item $I_C=I_{L\cup C_0}$, $L\nsubseteq \langle C_0\rangle$, $\{p\}=C\cap L \subset \sing(C_0)$ for a line $L$ and a cubic curve $C_0$;
\item $I_C=\langle x^2,xy,xq_1+yq_2\rangle$ after a change of coordinates, and $p$ is a point on $C$ such that $p \in L$ and $q_1(p)=q_2(p)=0$.
\end{enumerate}
\end{enumerate}
\end{proposition}

\noindent
In the cases (1)(c) and (2)(b), we frequently write $C=L^2Q$,
where $Q$ is the conic defined by $\langle x, q_2 \rangle$ and $L$ is the line $\langle x,y\rangle$,
because $C$ is a union of $Q$ and a double line supported on $L$.

\begin{lemma}
\label{lem:e1}
Let $E\subset \Hilb_{\PP^3}(4m)$ be the locus of connected locally CM-curves.
Let $\C=\{(p,C) \mid C\in E, p\in C\}$ be the universal family of $E$. Then $\C$ is an irreducible variety of dimension $17$.
Consequently, $\bE_1^{\infty}$ is an irreducible variety of dimension $17$.
\end{lemma}

\begin{proof}
The projection map $\pi \colon \C\subset E\times \PP^3 \to E$ is flat.
By \cite{vainsencher}, we know that $E$ is an irreducible variety. Now we apply \cite[III. Proposition 9.6]{Hartshorne}.
For each $e\in E$, every irreducible component of the fiber $\pi^{-1}(e)$ has dimension one.
Thus, every irreducible component of $\C$ has dimension $16+1=17$.
By Proposition \ref{prop:hilb}, $E$ is a blown-up space of a Grassmannian variety.
Thus, the inverse image of $\pi$ away from the exceptional locus in $E$ is irreducible.
But the inverse image of the exceptional locus in $E$ has dimension $16$, hence it does not form a new irreducible component of $\C$.
\end{proof}

\section{Proof of the main theorem}\label{sec:main}

We use the wall crossing of $\alpha$-stable pairs to relate $\M^{\infty}(4m+1)$ with $\M^{0^+}(4m+1)$.
Note that, since a sheaf in $\M(4m+1)$ has at least one non-zero section, the forgetful map from $\M^{0^+}(4m+1)$ to $\M(4m+1)$ is surjective.
Moreover, we have the following.

\begin{lemma}\label{lem:h0}
For $F\in \M(4m+1)$, we have $1\le h^0(F)\le 2$. Moreover if $h^0(F)=2$, then $F$ must be planar.
\end{lemma}

\begin{proof}
This is clear from the possible resolutions of $F$ found at \S \ref{sec:res}.
\end{proof}

\noindent
When $F$ is nonplanar, by choosing the unique non-zero section, we will regard a sheaf as a pair.
The locus of planar sheaves was studied in \cite{cc1}.

We will use the following well-known lemma frequently.
\begin{lemma}
Let $X$ be a projective scheme and $Y \subset X$ a closed subscheme.
Let $\F$ be a coherent $\O_X$-module and let $\G$ be a coherent $\O_Y$-module.
Then there is an exact sequence of vector spaces
\begin{multline}
\label{thomas1}
0 \to \Ext^1_{O_Y}(\F_{|Y}, \G) \to \Ext^1_{\O_X}(\F, \G) \to \Hom_{\O_Y}(\Tor_1^{\O_X}(\F, \O_Y), \G) \\
\to \Ext^2_{\O_Y}(\F_{|Y}, \G) \to \Ext^2_{\O_X}(\F, \G)
\end{multline}
In particular, if $\F$ is an $\O_Y$-module, then there is an exact sequence
\begin{multline}
\label{thomas2}
0 \to \Ext^1_{O_Y}(\F_{|Y}, \G) \to \Ext^1_{\O_X}(\F, \G) \to \Hom_{\O_Y}(\F \tensor_{\O_X} \I_Y, \G) \\
\to \Ext^2_{\O_Y}(\F_{|Y}, \G) \to \Ext^2_{\O_X}(\F, \G).
\end{multline}
\end{lemma}
\begin{proof}
The results directly come from the existence of the base change spectral sequence (\cite[Theorem 12.1]{mcc01}):
\[
\Ext_{\O_Y}^p(\Tor_q^{\O_X}(\F, \O_Y), \G) \Rightarrow \Ext^{p+q}_{\O_X}(\F, \G). \qedhere
\]
\end{proof}

\begin{remark}
In view of the above lemma, the results in \cite[Theorem 1.1]{freiermuth_trautmann} can be read as follows.
Using the wall-crossing (Remark \ref{3mplusonewall}), one can see that $\M(3m+1)$ consists of two components:
\begin{enumerate}
\item Non-planar sheaves, i.e. the structure sheaves of CM-curves of degree 3;
\item Planar sheaves.
\end{enumerate}

The dimensions of each component are 12 and 13, respectively.
Moreover, it was shown in \cite[Theorem 1.1]{freiermuth_trautmann} that the intersection of these two components consists of singular planar sheaves,
i.e., planar sheaves that are not locally free on their support.
The space of CM-curves is quasi-projective and smooth, cf. \cite[\S 4]{ps}.
Let $F$ be a stable sheaf supported on a curve $C$ that is contained in a plane $H\subset \PP^3$. From \eqref{thomas2}, we have
\[
0 \to \Ext_H^1(F,F) \to \Ext_{\PP^3}^1(F,F) \to \Hom_H(F,F(1)) \to \Ext_H^2(F,F) = \Ext_H^0(F,F(-3))^*=0.
\]
By the Grothendieck-Riemann-Roch theorem, $\dim \Ext_H^1(F,F)=10$.
Now, if $F$ is a locally free $\O_C$-module, then
\[
\Hom_H(F,F(1)) \isom \H^0(C, F^*\otimes F(1)) \isom \H^0(\O_C(1)) \isom \CC^3,
\]
hence $F$ gives a smooth point in the moduli space $\M(3m+1)$.
If $F$ is not locally free on its support, then from the exact sequence
\[
\ses{\O_C}{F}{\CC_p}
\]
we deduce that $\dim \Hom(F,F(1))\leq 4$.
According to \cite[Theorem 1.1]{freiermuth_trautmann}, this is actually an equality, hence $F$ gives a singular point in $\M(3m+1)$.
\end{remark}

\noindent
The numerical type of wall of the spaces $\M^{\alpha}(4m+1)$ is given by the following lemma.

\begin{lemma}\label{wall}
The wall crossing on $\M^{\alpha}(4m+1)$ occurs at $\alpha=3$ with the Jordan-H\"older filtration
\[
(0,\O_L)\oplus (1,\O_{C_0})
\]
where $C_0$ is a planar cubic curve.
\end{lemma}

\begin{proof}
By numerical computation, the possible type of wall is given by
\[
(0,F_{m+1})\oplus (1,F_{3m}),
\]
where the subscripts indicate the Hilbert polynomials of the sheaves (\S \ref{wallsection}). Obviously, $F_{m+1} \isom \O_L$.
The sheaf $F_{3m}$ is planar because it has a section and is pure. Thus, $F_{3m}$ is isomorphic to the structure sheaf of a planar cubic curve (\cite[Page 18]{drezet_maican}).
\end{proof}

\noindent
By wall crossing, the pairs in $\M^{\infty}(4m+1)$ of the form
\begin{equation}\label{eq6}
\ses{(0,\O_L)}{(1,F)}{(1,\O_{C_0})}
\end{equation}
are modified into pairs in $\M^{0^+}(4m+1)$ of the form
\[
\ses{(0,\O_{C_0})}{(1,F)}{(1,\O_L)}.
\]
We call the set of such pairs the \emph{wall crossing locus}.

The immediate consequence of Lemma \ref{wall} is that $\bX^{\infty}$ is dropped after the wall crossing.
We denote by $\bR^+$, $\bE_1^+$, $\bE_2^+$ and $\bP^+$ the loci in $\M^{0^+}(4m + 1)$ corresponding to
$\bR^\infty$, $\bE_1^\infty$, $\bE_2^\infty$, respectively, $\bP^\infty$.
Away from the wall crossing locus, the moduli space is unchanged.
By Lemma \ref{lem:h0}, after forgetting the section, we have $\bR^+ \isom \bR$, $\bE_1^+ \isom \bE_1$, and $\bE_2^+\isom \bE_2$.

\begin{lemma}
\hfill
\begin{enumerate}
\item The locus $\bR^{\infty}$ does not intersect the wall crossing locus. Thus, $\bR^{\infty} \isom \bR^{+}$.
\item The locus $\bP^{+}$ is the relative $0^+$-stable pairs space $\M^{0^+}(\PP \cU, 4m+1)$,
where $\cU$ is the universal rank three bundle over the Grassmannian $\operatorname{Gr}(3,4)=(\PP^3)^*$. In particular, $\bP^+$ is irreducible.
\end{enumerate}
\end{lemma}

\begin{proof}
The structure sheaf $\O_C$ of a connected CM-rational quartic curve $C\in \bR^{\infty}$ does not fits into the short exact sequence \eqref{eq6} because of $h^0(\O_C)=1$, hence (1) follows.
For (2), we can apply the wall crossing of \cite{cc1}.
Note that, by the results of \cite[Proposition 4.4]{cc1}, $\M_{\PP^2}^{0^+}( 4m+1)$ is a blow-up of $\M_{\PP^2}(4m+1)$ along the Brill-Noether locus, hence it is irreducible.
It follows that $\bP^+$ is irreducible.
\end{proof}

\noindent
Thus far we have shown that $\bR^{+}$ and $\bP^{+}$ are irreducible components of $\M^{0^+}(4m+1)$.
We denote the non-planar wall crossing loci by $W^\infty$ and $W^+$ respectively.
We will next show that $\M^{0^+}(4m+1)\setminus (\bR^{+} \cup \bP^{+})$ is contained in the closure of $\bE_1^{+}\setminus W^+$,
denoted $\bE^+$, which is irreducible by virtue of Lemma \ref{lem:e1}.

Recall that in Proposition \ref{prop:einf}, we decomposed $\bE^\infty_2$ into two subsets
according to whether the support of the sheaf is the union of a line with a cubic curve or the union of a double line with a conic curve.
After wall crossing, we also have the two corresponding loci, which we denote $\bE_{2a}^+$ and $\bE_{2b}^+$.

\begin{proposition}\label{prop:E2}
We have the decomposition $\bE_2^{+}\setminus W^{+} = \bE_{2a}^+ \cup \bE_{2b}^+$, where
\begin{enumerate}
\item $\bE_{2a}^+$ is the union of sets of the form
$\PP(\Ext^1(\CC_p, \O_C)) \setminus \PP(\Ext^1(\CC_p, \O_L(-1))) \isom \PP^1 \setminus \{ \text{\emph{pt}} \}$,
with $C = L \cup C_0$ for a cubic curve $C_0$, $L\nsubseteq \langle C_0\rangle$, $\{ p \}=C \cap L \subset \sing(C_0)$;
\item $\bE_{2b}^+$ is the union of sets of the form
$\PP(\Ext^1(\CC_p, \O_C)) \setminus \PP(\Ext^1(\CC_p, \O_L(-1))) \isom \PP^1 \setminus \{ \text{\emph{pt}} \}$,
with $C = L^2 \cup Q$ for a conic curve $Q$, $L\subset  \langle Q \rangle$, $p \in L$, $p_1(p)=p_2(p)=0$.
\end{enumerate}

\end{proposition}
\begin{proof}
Consider the exact sequence
\[
0=\Ext^0(\CC_p, \O_{C_0}) \to \Ext^1(\CC_p, \O_L(-1)) \isom \CC \to \Ext^1(\CC_p, \O_C) \xrightarrow{\delta} \Ext^1(\CC_p, \O_{C_0}) \isom \CC.
\]
If $F \in \PP(\Ext^1(\CC_p, \O_C))$ has non-zero image $G = \delta(F)$, then $F$ fits into the non-split short exact sequence
\[
\ses{\O_L(-1)}{F}{G}.
\]
By below Lemma \ref{building}, $F$ is a stable sheaf and, thus, it does not contain $\O_L$.
\end{proof}
%\noindent
%To study the relation between $W^+$ and the irreducible components of $\M^{0^+}(4m + 1)$ we need the following lemma.
\begin{lemma}\label{building}
Let $G$ be a stable sheaf with Hilbert polynomial $(d-1)m+1$. Let $L$ be a line.
Then, every sheaf $F$ fitting into the non-split short exact sequence
\begin{equation}
\ses{\O_L(-1)}{F}{G}
\end{equation}
is stable with Hilbert polynomial $dm+1$.
\end{lemma}

\begin{proof}
Let $F'$ be a subsheaf of $F$. Let $G'$ be its image in $\G$ and $K = F' \cap \O_L(-1)$.
We have $\chi(K) \le 0$. If $G' \neq G$, then $\chi(G') \le 0$, hence $\chi(F') \le 0$.
If $G' = G$, then $K \neq 0$ and $K \neq \O_L(-1)$, hence $\chi(K) \le -1$, hence $\chi(F') \le 0$.
\end{proof}

%\noindent
%This construction also works in flat families, as we will see below.

%\begin{lemma}
%Let $\mathcal{G}$ be the universal family of stable sheaves on $\M(3m+1)\times \PP^3$.
%Let $\mathcal{L}$ be the universal family of stable sheaves on $\M(m)\times \PP^3$.
%Then the universal extension $\mathcal{F}$ obtained by the relative projectivization
%\[
%K:=\PP(Ext^1(\mathcal{G}, \mathcal{L}))
%\]
%over $\M(3m+1)\times \M(m) \times \PP^3$ defines a natural embedding map $i \colon K\hookrightarrow \M(4m+1)$ whenever it is well-defined.
%\end{lemma}

%\begin{proof}
%For the existence, in a more general context, of the universal extension $\mathcal{F}$, see \cite{tomm}.
%The tangent map of $i$, that is, the Kodaira-Spencer map, is easily seen to be injective, hence $i$ is an embedding.
%\end{proof}

\noindent
Next we will show that the wall crossing locus is contained in $\bE^+$.
%For this we observe, firstly, that every stable sheaf in $W^+$ is contained in the above tautological family $K$.
%We will describe the possible sheaves $G$ by taking into account how far they are from being locally free on their support.

\begin{lemma}
The variety $W^{+}\cup\bE_{2a}^+$ is irreducible of dimension $15$ and is contained in $\bE^+$.
\end{lemma}

\begin{proof}
Consider the extension
\begin{equation}
\label{eq4}
\ses{\O_{C_0}}{F}{\O_L}.
\end{equation}
Denote $\{p\} = L \cap {C_0}$.
Let $H$ be the plane containing ${C_0}$. Tensoring \eqref{eq4} with $\O_H$ we get the exact sequence
\[
0 = \Tor_1^{\O_{\PP^3}} (\O_L, \O_H) \to \O_{C_0} \to \F_{| H} \to \CC_{p} \to 0
\]
from which we see that $\F_{| H} \isom \O_{C_0}(p)$ because $\Ext_H^1(\CC_p,\O_{C_0})=\CC$. We obtain the extension
\begin{equation}
\label{eq5}
0 \to \O_L(-1) \to F \to \O_{C_0}(p) \to 0.
\end{equation}
Conversely, by Lemma \ref{building}, any sheaf of this form is stable.
Let
\[
T = \{(p,{C_0},L) \mid \{p\}=L\cap {C_0}\quad \text{and} \quad L \nsubseteq H=\langle C_0 \rangle \}
\]
where ${C_0}$ is a plannar cubic curve contained in the plane $H$.
Consider the subset $T_{\sing} \subset T$ given by the condition $p\in \sing({C_0})$.
Obviously, $T$ is a $(\PP^2 \setminus \PP^1)$-bundle over the universal planar cubic curve.
Here $\PP^2 \setminus \PP^1$ parametrizes the choice of line $L$.
The space $T_{\sing}$ has the same bundle structure over the singular cubic curves,
hence $T$ and $T_{\sing}$ are smooth of dimension $15$, respectively, $13$.
%Let
%\[
%\phi \colon W^{+}\cup\bE_{2a}^+ \to T
%\]
%be the natural morphism sending an extension $F$ as in (\ref{eq5}) to $(p,{C_0},L)$.
%We claim that $\phi$ is a blow-up morphism $T$ along $T_{\sing}$. This shows that $W^{+}\cup\bE_{2a}^+$ is irreducible.
%By duality (\cite{maican_duality}),
%\[
%\Ext_{\PP^3}^1(\O_{C_0}(p),\O_L(-1)) \isom \Ext_{\PP^3}^1(\O_L(-1),\O_{C_0}(p)^\dual).
%\]
%Note that $\O_{C_0}(p)^\dual \isom I_{p,{C_0}}$. From \eqref{thomas1}, we obtain the exact sequence
%\[
%0 \to \Ext_{H}^1(\O_L(-1)|_H, I_{p,{C_0}}) \xrightarrow{\isom} \Ext_{\PP^3}^1(\O_L(-1),I_{p,{C_0}}) \to  \Hom (\CC_p, I_{p,{C_0}})=0.
%\]
%But
%\[
%\Ext_{H}^1(\O_L(-1)|_H, I_{p,{C_0}})^* \isom \Ext_{H}^1(I_{p,{C_0}}, \CC_p) \isom \Hom_{C_0}(I_{p,{C_0}}, \CC_p) \isom \Ext_{C_0}^1(\CC_p, \CC_p) \subset \Ext_H^1(\CC_p, \CC_p)
%\]
%where the first isomorphism is given by Serre duality (\cite[Theorem 3.12]{Huy06}).
%The second isomorphism is given by the exact sequence \eqref{thomas1}.

We claim that $W^{+}\cup\bE_{2a}^+$ is irreducible. Let us consider a $\PP^1$-bundle over $T$ defined by a relative extension sheaf.
%By construction, there exists an embedding
%\[
%T\subset \PP^3 \times Gr(2,4)\times \PP(\mathrm{Sym}^3 \cU),
%\]
%where $U$ is the universal rank $3$-bundle over $Gr(3,4)=(\PP^3)^*$.
Let us denote the natural projection maps:
\[
q_1:T\times \PP^3 \lr \PP^3\times \PP^3, \quad q_2:T\times \PP^3\lr Gr(2,4)\times \PP^3,
\]
\[
q_3:T\times \PP^3\lr \PP(\mathrm{Sym}^3\cU)\times \PP^3, \quad p:T\times \PP^3\lr T,
\]
where $U$ is the universal rank $3$-bundle over $Gr(3,4)=(\PP^3)^*$.
Let $\cF_1$, $\cF_2$, and $\cF_3$ be the universal families of points, lines, and planar cubics, respectively.
Let
\[
\pi: P:=\PP(\cExt_p^1(q_1^*\cF_1,q_2^*\cF_2(-1)\oplus q_3^*\cF_3)) \rightarrow T
\]
be the structure morphism of the projectivized relative extension sheaf over $T$. Then the tautological family $\cF$ parameterized by $P$ (\cite[Example 2.1.12]{huybrechts_lehn})
%fits into the following exact sequence:
%\[
%\ses{(\pi\times \mathrm{id})^*(\cF_2(-1)\oplus \cF_3)\otimes \O_P(1)}{\cF}{(\pi\times \mathrm{id})^*\cF_1}
%\] which
define a rational map
\[
\Psi:P\dashrightarrow W^{+}\cup\bE_{2a}^+.
\]
By some diagram chasing, the map $\Psi$ is well-defined on the complement of the union of two disjoint sections $s(T)\cup s'(T)\subset P$ consisting of sheaves of the forms $\{\O_L(-1)\oplus \O_{C_0}(p)\} \cup \{\O_L\oplus \O_{C_0}\}$ because it has the destabilizing quotient sheaf $\O_L(-1)$ and $\O_{C_0}$ respectively. Because of the existence of relative Quot scheme, it forms a flat family of sheaves $\cQ$ over $s(T)$. Let $g:\widetilde{P}\rightarrow P$ be the blowing-up of $P$ along $s(T)$. Let $\widetilde{s(T)}$ be the exceptional divisor.
Let
\[
\cF':=\mathrm{ker}(\widetilde{\cF}\twoheadrightarrow \widetilde{\cF}|_{\widetilde{s(T)}\times \PP^3} \twoheadrightarrow \widetilde{\cQ})
\]
be the composition of the surjective maps where $\widetilde{F}$ denotes the pull-back of the sheaf $F$ along the map $g\times \mathrm{id}:\widetilde{P}\times \PP^3 \rightarrow P\times\PP^3$.
Then the modification has the effect of the change of the subsheaf and quotient sheaf (cf. \cite[Definition 2.5]{cc1}). To finish the proof of our claim, it is enough to show that all stable sheaves parameterized by $W^{+}\cup\bE_{2a}^+$ aries from this modification. That is, let us show that the normal space
\[
N_{s(T)/P,s}\twoheadrightarrow \Ext_{\PP^3}^1(\O_{C_0}(p),\O_L(-1))
\]
is surjective where $s:=g(v)=[\O_L(-1)\oplus \O_{C_0}(p)]$ for $v\in \widetilde{s(T)}$. Note that there is an inclusion $N_{T_{\mathrm{sing}}/T,\pi(s)}\subset N_{s(T)/P,s}$ and the former space is isomorphic to $\Ext_{C_0}^1(\CC_p, \CC_p)^*$ (\cite[\S 3.1.1]{iena1}).
Also,
\[
\Ext_{C_0}^1(\CC_p, \CC_p)^*\cong \Hom_{C_0}(I_{p,C_0}, \CC_p)^*\cong \Ext^1_{C_0}(\CC_p,I_{p,C_0})
\]
\[
\cong \Ext_H^1(\CC_p,I_{p,C_0})\cong \Ext^1_{\PP^3}(\O_L(-1), I_{p,C_0}) \cong \Ext_{\PP^3}^1(\O_{C_0}(p),\O_L(-1))
\]
where the first one comes from $\ses{I_{p,C_0}}{\O_{C_0}}{\CC_p}$. The second one comes from the fact that the dualizing sheaf $\omega_{C_0}$ is a line bundle and thus one can apply a version of Serre duality. The third and fourth one come from the equation \eqref{thomas1} and $\O_L(-1)|_H\cong \CC_p$. The last one is \cite[Theorem 13]{maican_duality}. Hence we proved that there exists a surjective (rational) map
\[
\widetilde{\Psi}:\widetilde{P}\rightarrow W^{+}\cup\bE_{2a}^+
\]
which finish the proof of our claim.
%The space $\Ext_{C_0}^1(\CC_p, \CC_p)$ can be naturally identified with the tangent space of ${C_0}$ at $p$,
%hence it is naturally embedded into the tangent space $\Ext_H^1(\CC_p, \CC_p)$ of $H$ at $p$.
%More precisely, $\Ext_{C_0}^1(\CC_p, \CC_p)$ is isomorphic to $\CC$ if $p$ is a smooth point of ${C_0}$ and to $\CC^2$ otherwise.
%Thus,  $\phi$ is an isomorphism away from $T_{\sing}$, hence $\phi$ is a birational map.
%Furthermore, by \cite[\S 3.1.1]{iena1},
%one can identify the normal space of $T_{\sing}$ in $T$ with the extension group $\Ext_{C_0}^1(\CC_p, \CC_p)$,
%which is classifying the fiber of $\phi$ (up to a scalar). This implies that the morphism $\phi$ is a blow-up $T$ along $T_{\sing}$.
As general points in $W^{+}\cup\bE_{2a}^+$ are contained in $\bE^+$ and since it is irreducible,
$W^{+}\cup\bE_{2a}^+$ itself is contained in $\bE^+$.
\end{proof}

\noindent
It remains to show that $\bE_{2b}^\infty\setminus W^{\infty}=\bE_{2b}^+\setminus W^{+}$ is contained in  $\bE^+$.
To this end, let $Z$ be the locus of the extension sheaves $F$ fitting into the short exact sequence
\[
\ses{\O_L(-1)}{F}{O_{LQ}(p)}
\]
such that $p\in Q$. Clearly, $\bE_{2b}^+$ is contained in $Z$. We will show that $Z$ is irreducible.
Since general elements in $Z$ are contained in $\bE^+$, this proves our assertion.
Consider the map
\[
\phi \colon Z \to \Hilb_{\PP^2}(m+1) \times \M_{\PP^2}(3m+1), \qquad [\F] \mapsto (L, [\O_{LQ}(p)]).
\]
It will turn out that $\phi$ is a projective bundle over some irreducible variety. This will prove that $Z$ is an irreducible variety.

\begin{lemma}
\label{lemma:S}
The locally closed subset $\mathbf{S} \subset \Hilb_{\PP^2}(m+1) \times \M_{\PP^2}(3m+1)$
of pairs $(L, [\O_{C_0}(p)])$ for which $C_0= L \cup Q$ and $p \in Q$ for a conic curve $Q \subset \PP^2$ is irreducible.
\end{lemma}

\begin{proof}
Let $\M_{\PP^2}^{\infty}(2m+2)$ be the moduli space of pairs with Hilbert polynomial $2m+2$.
One can easily see that it is isomorphic to the universal conic curve, which is a $\PP^4$-bundle over $\PP^2$,
so is irreducible (\cite[Lemma 2.3]{cc1}). The morphism
\[
\Hilb_{\PP^2}(m+1) \times \M_{\PP^2}^{\infty}(2m+2) \to \mathbf{S}, \qquad (L,p,Q)\mapsto (L, \O_{Q\cup L}(p))
\]
is well-defined and surjective. Thus, $\mathbf{S}$ is irreducible.
\end{proof}

\begin{proposition}
The locus $Z$ is irreducible of dimension $14$.
\end{proposition}

\begin{proof}
In view of Lemma \ref{lemma:S}, it is enough to show that the morphism $Z \to \mathbf{S}$ is surjective
and that its fibers are irreducible of the same dimension.
We will prove that
\[
\Ext^1_{\O_{\PP^3}}(\O_{C_0}(p), \O_L(-1)) \isom \CC^4
\]
for all $(L, [\O_{C_0}(p)]) \in \mathbf{S}$.
From (\ref{thomas1}) we have the exact sequence
\begin{multline*}
0 \to \Ext^1_{\O_L}(\O_{C_0}(p)_{| L}, \O_L(-1)) \to \Ext^1_{\O_{\PP^3}}(\O_{C_0}(p), \O_L(-1)) \to \\
\Hom(\Tor_1^{\O_{\PP^3}}(\O_{C_0}(p), \O_L), \O_L(-1)) \to \Ext^2_{\O_L}(\O_{C_0}(p)_{| L}, \O_L(-1)) = 0
\end{multline*}
Assume firstly that $p \notin L$.
The long exact sequence of torsion sheaves associated to the short exact sequence
\[
0 \to \O_{C_0} \to \O_{C_0}(p) \to \CC_p \to 0
\]
yields the isomorphisms
\[
\O_{C_0}(p)_{| L} \isom {\O_{C_0}}_{| L} \isom \O_L, \qquad
\Tor_1^{\O_{\PP^3}}(\O_{C_0}(p), \O_L) \isom \Tor_1^{\O_{\PP^3}}(\O_{C_0}, \O_L) \isom \O_L(-1) \oplus \O_L(-3).
\]
We obtain the isomorphisms
\[
\Ext^1_{\O_L}(\O_{C_0}(p)_{| L}, \O_L(-1)) = 0, \qquad \Hom(\Tor_1^{\O_{\PP^3}}(\O_{C_0}(p), \O_L), \O_L(-1)) \isom \CC^4.
\]
Assume now that $p \in L$. We have the exact sequence
\[
0 \to 2\O(-3) \xrightarrow{\delta} 3\O(-2) \oplus \O(-1) \xrightarrow{\gamma} \O(-1) \oplus \O \to \O_{C_0}(p) \to 0
\]
\[
\delta = \left[
\ba{ll}
\phantom{-} x & \phantom{-} 0 \\
\phantom{-} 0 & \phantom{-} x \\
-y & -z \\
\phantom{-} 0 & -q_2
\ea
\right], \qquad \gamma = \left[
\ba{cccc}
y & z & x & 0 \\
0 & q_2 & 0 & x
\ea
\right]
\]
where $p$ is given by the ideal $\langle x, y, z \rangle$, $L$ is given by the ideal $\langle x, y \rangle$,
and $Q$ is given by the ideal $\langle x, q_2 \rangle$ (\cite[Proposition 3.5]{freiermuth_trautmann}).
Tensoring with $\O_L$, shows that $\O_C(p)_{| L}$ is isomorphic to the cokernel of the morphism
\[
\O_L(-2) \xrightarrow{\tiny \left[ \!\! \ba{c} {z}_{| L} \\ {q_2}_{| L} \ea \!\! \right]} \O_L(-1) \oplus \O_L
\]
and that $\Tor_1^{\O_{\PP^3}}(\O_{C_0}(p), \O_L)$ is isomorphic to the middle cohomology of the sequence
\[
2 \O_L(-3) \xrightarrow{\delta_L} 3\O_L(-2) \oplus \O_L(-1) \xrightarrow{\gamma_L} \O_L(-1) \oplus \O_L
\]
\[
\delta_L = \left[
\ba{ll}
0 & \phantom{-} 0 \\
0 & \phantom{-} 0 \\
0 & {-z}_{| L} \\
0 & {-q_2}_{| L}
\ea
\right], \qquad \gamma_L = \left[
\ba{cccc}
0 & {z}_{| L} & 0 & 0 \\
0 & {q_2}_{| L} & 0 & 0
\ea
\right],
\]
which is isomorphic to the cokernel of the morphism
\[
\O_L(-3) \xrightarrow{\tiny \left[ \!\! \ba{c} 0 \\ {z}_{| L} \\ {q_2}_{| L} \ea \!\! \right]} 2\O_L(-2) \oplus \O_L(-1).
\]
By hypothesis, $p$ is a point on $Q$, hence ${z}_{| L}$ divides ${q_2}_{| L}$.
It now becomes clear that
\[
\O_{C_0}(p)_{| L} \isom \CC_p \oplus \O_L \quad \text{and} \quad
\Tor_1^{\O_{\PP^3}}(\O_{C_0}(p), \O_L) \isom \CC_p \oplus \O_L(-2) \oplus \O_L(-1).
\]
We obtain the isomorphisms
\[
\Ext^1_{\O_L}(\O_{C_0}(p)_{| L}, \O_L(-1)) \isom \CC, \qquad \Hom(\Tor_1^{\O_{\PP^3}}(\O_{C_0}(p), \O_L), \O_L(-1)) \isom \CC^3.
\qedhere
\]
\end{proof}

\noindent
Summarizing the results obtained thus far, we have the following.

\begin{proposition}
The space $\M^{0^+}(4m+1)$ consists of three irreducible components $\bR^{+}$, $\bE^{+}$ and $\bP^{+}$.
\end{proposition}

\noindent
By forgetting the section, we arrive at our main theorem.

\begin{theorem}\label{mainthm}
The moduli space $\M_{\PP^3}(4m+1)$ consists of three irreducible components $\bR$, $\bE$ and $\bP$,
whose dimensions are 16, 17 and 20, respectively.
\end{theorem}

\section{The intersections of the irreducible components}
\label{sec:int}

In this section, we describe the intersections $\bR\cap \bP$ and $\bE\cap \bP$,
and we give a non-exhaustive list of sheaves in $\bR\cap \bE$.

\subsection{The intersections $\bR\cap \bP$ and $\bE\cap \bP$}

We begin by collecting some known facts about the relation of the space of finite maps to the component $\bR$.
When a finite map is birational to its image, we call it a \emph{birational finite map}.
By the following proposition,
the birational finite maps give sheaves in the boundary of the locus of structure sheaves of smooth rational quartic curves.

\begin{proposition}\label{prop:dawei}
Let $F_{0}(\PP^r,d)$ be the space of finite maps from a genus $0$ curve of degree $d$ to $\PP^r$. Then,
\begin{enumerate}
\item The space $F_0(\PP^r,d)$ is an irreducible projective variety of dimension $(d+1)(r+1)-3$.
\item Consider the rational map
\[
\Psi \colon F_{0}(\PP^r,d) \dashrightarrow \M_{\PP^r}(dm+1), \qquad
[f \colon C\rightarrow \PP^r]\mapsto f_*\O_C.
\]
Then $\Psi$ is injective over the locus of birational finite maps.
\end{enumerate}
\end{proposition}
\begin{proof}
By \cite[Theorem 2]{FP97},
the moduli space $\overline{M}_{0}(\PP^r,d)$ of stable maps of genus $0$ and degree $d$ is an irreducible variety.
Also, there exists a small contraction morphism $\overline{M}_{0}(\PP^r,d) \to F_{0}(\PP^r,d)$ (\cite[Proposition 3.11]{CCC11}).
This proves the first claim. Part (2) follows directly from the proof of \cite[Proposition 3.18]{CCC11}.
\end{proof}

\begin{proposition}
\label{thm:RcapP}
Let $H \subset \PP^3$ be a plane and let $C \subset H$ be an irreducible quartic curve
having three distinct nodal singular points $P_1$, $P_2$, $P_3$.
Then the unique extension of $\CC_{P_1} \oplus \CC_{P_2} \oplus \CC_{P_3}$ by $\O_C$,
denoted $\O_C(P_1 + P_2 + P_3)$, gives a point in $\bR \cap \bP$.
We denote by $(\bR \cap \bP)_0$ the set of such sheaves.
The intersection $\bR \cap \bP$ is irreducible and is the closure of $(\bR\cap \bP)_0$.
\end{proposition}

\begin{proof}

Let $v \colon \PP^1 \to C$ be the normalization map.
%Since the degree of $C$ is $4$, $\deg(v)=4$,
Since the degree $\deg(v)=4$, hence $v$ is an element of the moduli space $F_{0}(\PP^3,4)$ of finite maps of degree $4$.
By Proposition \ref{prop:dawei}, $F_{0}(\PP^3,4)$ is an irreducible variety and the locus $\mathrm{Map}^{sm}(\PP^1,\PP^3)/PGL(2)$ of finite maps whose image is smooth is a Zariski dense open subset.
One can choose a one-parameter family of finite maps $i_t \colon C_t \to \PP^3$
such that the general fibers are smooth rational quartic curves and $\lim_{t \to 0}\O_{C_t} = v_*\O_{\PP^1}$.
But by part (2) of Proposition \ref{prop:dawei}, we know that the direct image sheaf $v_*\O_{\PP^1}$ is stable. This is exactly the unique stable extension given by the normalization sequence:
\[
\ses{\O_C}{v_*O_{\PP^1}}{\CC_{P_1}\oplus \CC_{P_2}\oplus \CC_{P_3}}.
\]
We deduce that $\O_C(P_1 + P_2 + P_3)$ lies in $\bR \cap \bP$.

Assume that $\F$ gives a point in $\bR \cap \bP$, where $\operatorname{supp}(\F) \subset H$.
We have $F = \lim_{t \to 0}\O_{C_t}$ for smooth rational quartic curves $C_t$, $t \neq 0$.
Let us choose a projection center $O\in \PP^3$ that does not lie on $H$ or on any of the curves $C_t$.
Let $\pi \colon \PP^3 \setminus \{ O \} \to H$ be the projection with center $O$.
Then $\pi_*\O_{C_t}$ are stable sheaves for $t\neq 0$ by the argument of the previous paragraph.
Also, $\pi_*\F=\F$ because $\operatorname{supp}(\F) \subset H$.
Thus, $\lim_{t \to 0}(\pi_*\O_{C_t})=\pi_*(\lim_{t \to 0}\O_{C_t})=\F$.
Since $\pi_*\O_{C_t} \in (\bR \cap \bP)_0$,
$\F$ is an element of the closure of $(\bR \cap \bP)_0$.

%Let $\Gamma \subset \overline{\RR}$ be an irreducible quasi-projective curve with associate flat family
%$\{ \F_s \}_{s \in \Gamma}$. Assume that $[\F]$ belongs to $\Gamma$ and $\Gamma \setminus \{ [\F] \} \subset \RR_0$.
%Choose a point $O \in \PP^3 \setminus H$ that is not on the support of $\F_s$ for all $s \in \Gamma$.
%Let $\pi \colon \PP^3 \setminus \{ O \} \to H$ be the projection with center $O$.
%We claim that the family $\{ \pi_*(\F_s) \}_{s \in \Gamma}$ on $H$ is flat relative to $\Gamma$.
%Indeed, let $\R$ be a coherent $\Gamma$-flat sheaf on $\Gamma \times \PP^3$ that restricts to $\F_s$
%on each fiber $\{ s \} \times \PP^3$. Arguing as before, we can show that the sheaf
%$(\id \times \pi)_* (\R)$ on $\Gamma \times H$ is flat relative to $\Gamma$
%and parametrizes the family $\{ \pi_*(\F_s) \}_{s \in \Gamma}$.
%We saw above that $[\pi_*(\F_s)]$ lies in the closure of $(\overline{\RR} \cap \mathbf{P})_0$ if $s \neq [\F]$.
%Since $\pi_*(\F) \isom \F$, we conclude that $[\F]$ belongs to the closure of $(\overline{\RR} \cap \mathbf{P})_0$, as well.
%This proves that $\overline{\RR} \cap \mathbf{P}$ is the closure of $(\overline{\RR} \cap \mathbf{P})_0$.

Obviously, the open part $(\bR \cap \bP)_0$ of the intersection has a fibration structure over the Grassmannian $Gr(3,4)=(\PP^3)^*$.
The fibres are birational with the irreducible variety $F_{0}(\PP^2,4)$ of finite maps.
Thus, $(\bR \cap \bP)_0$ is irreducible.
%Take two such sheaves $\O_C(P_1 + P_2 + P_3)$ and $\O_{C'}(P_1' + P_2' + P_3')$.
%Choose a point $O \in \PP^3 \setminus H$
%and denote by $\pi \colon \PP^3 \setminus \{ O \} \to H$ the projection with center $O$.
%We saw above that there are smooth rational quartic curves $R, R' \subset \PP^3 \setminus \{ O \}$
%such that $\pi_*(\O_R) \isom \O_C(P_1 + P_2 + P_3)$ and $\pi_*(\O_{R'}) \isom \O_{C'}(P_1' + P_2' + P_3')$.
%Let $\Gamma \subset \RR_0$ be an irreducible quasi-projective curve containing $[\O_R]$ and $[\O_{R'}]$
%with associated flat family $\{ \O_{R_s} \}_{s \in \Gamma}$.
%We may assume that $O \notin R_s$ for all $s \in \Gamma$.
%As before, the family $\{ \pi_*(\O_{R_s}) \}_{s \in \Gamma}$ is flat over $\Gamma$,
%so it corresponds to a map $f \colon \Gamma \to \overline{\RR} \cap \M_H(4m+1)$.
%Thus both $[\O_C(P_1 + P_2 + P_3)]$ and $[\O_{C'}(P_1' + P_2' + P_3')]$ belong to $f(\Gamma)$, which is irreducible,
%so they belong to the same irreducible component of $\overline{\RR} \cap \M_H(4m+1)$.
%But $\overline{\RR} \cap \M_H(4m+1)$ is the closure of $(\overline{\RR} \cap \mathbf{P})_0 \cap \M_H(4m+1)$.
%We conclude that this set is irreducible.
\end{proof}

\begin{proposition}
For any $F\in \bR$, we have $h^0(F)=1$.
\end{proposition}

\begin{proof}
For any plane $H \subset \PP^3$ and any sheaf $\F$ giving a point in $\M_H (4m+1)$ the exact sequence (\ref{thomas2}) reads
\[
0 \to \Ext^1_{\O_H}(\F, \F) \isom \CC^{17} \to \Ext^1_{\O_{\PP^3}} (\F, \F) \to \Hom(\F(-1), \F) \to \Ext^2_{\O_H}(\F, \F) \isom \Hom(\F, \F(-3))^* = 0.
\]
Assume now that $\F$ gives a generic point in $\bR \cap \bP$ of the form $\O_C(P_1 + P_2 + P_3)$,
where $C$ is an irreducible planar quartic curve having distinct nodes at $P_1$, $P_2$, $P_3$.
Then $\Hom(F(-1),F) \isom \CC^5$. This shows that $\ext^1(\F, \F) \ge 22$. Since the embedding dimension is upper semicontinuous and $\bR \cap \bP$
is irreducible, this estimate holds for all sheaves in $\bR \cap \bP$.

Suppose now that $h^0(F)=2$.
Then we have $F \isom \O_C(-p)(1)$ for some planar quartic curve $C$ and a point $p \in C$ (\cite[Proposition 3.3.4]{drezet_maican}). Let $H$ be the plane containing $C$.
From the exact sequence
\[
0 \to \O_C(-p) \to \O_C \to \CC_p \to 0
\]
we get the exact sequence
\[
0 = \Hom (\CC_p, F) \to \Hom(\O_C, F) \isom \CC^2 \to \Hom(\O_C(-p), F) \to
\Ext^1_{\O_H}(\CC_p, F) \isom \Ext^1_{\O_H}(F, \CC_p)^*.
\]
From the resolution
\[
0 \to \O_H(-3) \oplus \O_H(-1) \to 2\O_H \to F \to 0,
\]
we see that $\Ext^1_{\O_H}(F, \CC_p)^*$  is isomorphic to $\CC$ or $\CC^2$, depending on whether $p$ is a regular or a singular point of $C$.

Thus, $\hom(F(-1), F) \le 4$. In view of the exact sequence at the beginning of the proof, we deduce that $\ext^1(F, F) \le 21$.
In view of the above estimate, it follows that $F \notin \bR \cap \bP$.
\end{proof}

\begin{remark}
More generally, let $R_d(\PP^r)$ be the space of irreducible rational curves in $\PP^r$ of degree $d$.
By using elementary modifications of sheaves (\cite[Theorem 2.B.1]{huybrechts_lehn}),
one can easily see that every stable sheaf $F\in \overline{R_d(\PP^r)}\subset \M_{\PP^r}(dm+1)$ satisfies $h^0(F)=1$.
\end{remark}

\noindent
Similar arguments as in Proposition \ref{thm:RcapP} yield the following.

\begin{proposition}\label{thm:EcapP}
Let $H \subset \PP^3$ be a plane and let $C \subset H$ be an irreducible quartic curve
having two distinct nodal singular points $P_1$ and $P_2$ and no other singularities. Let $P$ be a regular point of $C$.
Then $\O_C(P_1 + P_2 + P)$, gives a point in $\bE \cap \bP$.
We denote by $(\bE \cap \bP)_0$ the set of such sheaves.
The intersection $\bE \cap \bP$ is irreducible and is the closure of $(\bE \cap \bP)_0$.
\end{proposition}

\begin{remark}
Since $\bR\cap \bP\subset \bE \cap \bP$, it follows that $\bR \cap \bP \subset \bR \cap \bE$.
\end{remark}

\subsection{The intersection $\bR \cap \bE$}
The intersection $\bR \cap \bE$ is difficult to describe.
We provide below a list of general sheaves in $\bR \cap \bE$.
% by following our description of $\bE$ in \S\ref{sec:main}. %See also Proposition \ref{prop:einf}.
By Proposition \ref{prop:dawei}, all sheaves in $\bR$ are flat limits of sheaves of the form $f_*\O_{\tilde C}$ for finite maps $f$.
When $f$ is not birational, the flat limit can be described by the elementary modification technique, which we will repeatedly use in this section.
For the general setting of modification of sheaves, see \cite[Theorem 2.B.1]{huybrechts_lehn}.
For the version concerning maps, see \cite{CK11, CHK12}.

Our first order of business is to show that $W^+$ is disjoint from $\bR \cap \bE$.
Recall that $W^+$ is the locus of non-planar sheaves $F$ which fit into an exact sequence
\[\ses{\O_{C_0}}{F}{\O_L},\]
for a planar cubic curve ${C_0}$ and an incident line $L$ that is not coplanar with $C_0$.

\begin{proposition}
\label{prop:wsmooth}
Assume that $F$ gives a point in $W^+$.
Then $\ext^1(F, F) = 17$, so $F$ is a smooth point of $\bE$.
\end{proposition}

\begin{proof}
From \eqref{eq5}, we have the exact sequence
\[
\Ext^1 (\F, \O_L(-1)) \to \Ext^1(\F, \F) \to \Ext^1(\F, \O_{C_0}(p)).
\]
Denote $\{p\} = L \cap C_0$. From (\ref{thomas1}) we have the exact sequence
\begin{multline*}
0 \to \Ext^1_{\O_L}(\F_{| L}, \O_L(-1)) \to \Ext^1_{\O_{\PP^3}}(\F, \O_L(-1)) \to \Hom(\Tor_1^{\O_{\PP^3}}(\F, \O_L), \O_L(-1)) \to \\
\Ext^2_{\O_L}(\F_{| L}, \O_L(-1)) = 0.
\end{multline*}
From (\ref{eq4}) and (\ref{eq5}) we have the exact sequences
\[
\CC_p = {\O_{C_0}}_{| L} \to \F_{| L} \to \O_L \to 0 \quad \text{and} \quad
0 \to \O_L(-1) \to \F_{| L} \to \O_{C_0}(p)_{| L} \to 0.
\]
Note that $\O_{C_0}(p)_{| L}$ is isomorphic to $\CC_p$ if $p \in \reg(C_0)$, respectively,
to $\CC_p \oplus \CC_p$ if $p \in \sing(C_0)$.
It follows that $\F_{| L}$ is isomorphic to $\O_L$ if $p \in \reg(C_0)$, respectively, to $\CC_p \oplus \O_L$ if $p \in \sing(C_0)$.
From (\ref{eq4}) we get the long exact sequence of torsion sheaves of $\O_{\PP^3}$-modules
\[
\CC_p \isom \Tor_1(\O_{C_0}, \O_L) \to \Tor_1(\F, \O_L) \to \Tor_1(\O_L, \O_L) \isom 2\O_L(-1) \to \CC_p \to \F_{| L} \to \O_L \to 0.
\]
From this we see that
\[
\Tor_1(\F, \O_L)/\T^0(\Tor_1(\F, \O_L)) \isom
\begin{cases}
\O_L(-2) \oplus \O_L(-1) & \text{if $p \in \reg(C)$}, \\
2\O_L(-1) & \text{if $p \in \sing(C)$}.
\end{cases}
\]
In the first case we have the exact sequence
\[
0 = \Ext^1_{\O_L}(\F_{| L}, \O_L(-1)) \to \Ext^1_{\O_{\PP^3}}(\F, \O_L(-1)) \to \Hom(\O_L(-2) \oplus \O_L(-1), \O_L(-1)) \isom \CC^3 \to 0.
\]
In the second case we have the exact sequence
\[
0 \to \Ext^1_{\O_L}(\F_{| L}, \O_L(-1)) \isom \CC \to \Ext^1_{\O_{\PP^3}}(\F, \O_L(-1)) \to \Hom(2\O_L(-1), \O_L(-1)) \isom \CC^2 \to 0.
\]
In both cases we get the isomorphism $\Ext^1_{\O_{\PP^3}}(\F, \O_L(-1)) \isom \CC^3$.

Let $H$ be the plane containing $C_0$. From (\ref{thomas1}) and taking into account that $\F_{| H} \isom \O_{C_0}(p)$ we get the exact sequence
\begin{multline*}
0 \to \Ext^1_{\O_H} (\F_{| H}, \O_{C_0}(p)) \isom \CC^{10} \to \Ext^1_{\O_{\PP^3}}(\F, \O_{C_0}(p)) \to \Hom(\Tor_1^{\O_{\PP^3}}(\F, \O_H), \O_{C_0}(p)) \\
\to \Ext^2_{\O_H} (\F_{| H}, \O_{C_0}(P)) = 0.
\end{multline*}
The long exact sequence of torsion sheaves associated to (\ref{eq4}) reads in part
\[
0 = \Tor_2 (\O_L, \O_H) \to \Tor_1 (\O_{C_0}, \O_H) \isom \O_{C_0}(-1) \to \Tor_1(\F, \O_H) \to \Tor_1(\O_L, \O_H) = 0.
\]
We obtain the exact sequence
\[
0 \to \CC^{10} \to \Ext^1_{\O_{\PP^3}}(\F, \O_{C_0}(p)) \to \Hom(\O_{C_0}(-1), \O_{C_0}(p)) \to 0.
\]
Thus, $\Ext^1_{\O_{\PP^3}}(\F, \O_{C_0}(p)) \isom \CC^{14}$.
From the exact sequence at the beginning of the proof we get the inequality $\ext^1(\F, \F) \le 17$.
This must be an equality because the moduli space has dimension $17$ at $F$ (\cite[Corollary 4.5.2]{huybrechts_lehn}).
\end{proof}

\noindent
In view of Proposition \ref{prop:dawei}, whenever the reduced support of a sheaf $F \in \overline{\bR}$ is of degree 4,
$F \isom f_*\O_{\tilde C}$ for some birational finite map $f \colon \tilde{C}\to \PP^3$.
The list of such sheaves is as follows.
%These are general sheaves in the loci of sheaves (after wall-crossing) of Proposition \ref{prop:einf} (1)(a), (1)(b), and (2)(a).

\begin{proposition}
\label{prop:reddeg4}
Let $F \in \bR \cap \bE$ have reduced support of degree 4. Then one of following holds:
\begin{enumerate}
\item $F$ is the unique non-split extension of $\CC_p$ by $\O_C$, where either
\begin{enumerate}
\item $C$ has an ideal $\langle q_1, q_2\rangle$ generated by two quadratic polynomials and $p$ is a singular point of $C$; or
\item $C$ is the union of a planar cubic $C_0$ that is singular at $p$ and a line $L$ meeting $C_0$ at a regular point.
\end{enumerate}
\item When $C$ is the union of a planar cubic $C_0$ that has a node at $p$ and a line $L$ passing through $p$,
only two sheaves in the extension $\PP(\Ext^1(\CC_p,\O_C)) \isom \PP^1$ belong to $\bR \cap \bE$.
\item When $C$ is the union of a planar cubic $C_0$ that has a cusp at $p$ and a line $L$ passing though $p$,
precisely one sheaf in the extension $\PP(\Ext^1(\CC_p,\O_C)) \isom \PP^1$ belongs to $\bR \cap \bE$.
\end{enumerate}
The sheaves from (3) lie in the closure of the set of sheaves from (2).
\end{proposition}

\begin{proof}
Let $C$ be the reduced support of $F$. From Proposition \ref{prop:hilb},
we see that either $I_C=\langle q_1, q_2\rangle $ or $I_C = \langle xy, xz, yq_1 + zq_2\rangle$ after a change of coordinates,
where $q_1$ and $q_2$ are quadratic polynomials.
In order to lie in $\bR \cap \bE$, $F$ must be of the form $f_*\O_{\tilde C}$, with $p$ a singular point of $C$.
By the classification of sheaves in $\bE^{\infty}$ found at Proposition \ref{prop:einf},
we see that case (1) above occurs precisely when $\Ext^1(\CC_p, \O_C) \isom \CC$.
Indeed, we may take ${\tilde C}$ to be the partial normalization of $C$ at $p$.
The same argument applies in case (3).
In case (2) we may take two different partial normalizations of $C$ at $p$, each yielding a sheaf in $\bR \cap \bE$.
\end{proof}

\noindent
Note that, in case (2), by Proposition \ref{prop:E2}(1) there is a non-stable sheaf in the extension $\Ext^1(\CC_p, \O_C)$.
By the properness of the moduli space, the flat limit of sheaves in this extension must lie in $W^+$, and hence outside of $\bR$.
We will see below that in the case of Proposition \ref{prop:E2}(2) the flat limit of sheaves in the said extension is still in $\bR$.
This does not contradict Proposition \ref{prop:wsmooth} because in this case the flat limit is a planar sheaf and hence not in $W^+$.

\begin{corollary}
\label{non-reduced_degenerations}
In each of the following cases the sheaf $F$ gives a point in $\bR \cap \bE$:
\begin{enumerate}
\item $\F = \O_C(p)$, where $C$ is a singular curve of type $(2, 2)$ on a smooth quadric surface and $p \in \sing(C)$;
\item $\F = \O_C(p)$, where $C$ is a quadruple line supported on $L$
(the intersection of a double plane containing $L$ with the union of two distinct planes each containing $L$) and $p \in L$;
\item $\F = \O_C(p)$, where $C$ is the union of a singular planar curve $C_0$ with an incident line $L$, and $p \in \sing(C_0) \setminus L$.
\end{enumerate}
\end{corollary}

\begin{proof}
The sheaves from (1) are limits of sheaves from Proposition \ref{prop:reddeg4}(1)(a) because the set of singular curves of type $(2, 2)$
on $\PP^1 \times \PP^1$ is irreducible (\cite[Theorem 3.1]{Tyo07}). In particular, if $C$ is a double conic (meaning the intersection of a double plane with a smooth quadric $S$)
and $p \in C$, then $\O_C(p) \in \bR \cap \bE$. The sheaves from (2) are limits of such sheaves
(make $S$ converge to the union of two distinct planes).
Finally, the sheaves from (3) are limits of sheaves from Proposition \ref{prop:reddeg4}(1)(a).
\end{proof}

%Even though the reduced degree is $4$, there may be non-isomorphic stable sheaves in the intersection coming from the direct image sheaf of the choice of the normalization.
%The proof of (1) is achieved by taking the normalization of an irreducible elliptic curve with a nodal singularity
%and then arguing as at Theorem \ref{thm:RcapP}.
%
%
%*****Explain the direct image sheaf and its stability***

%\begin{proposition}
%Let $F \in \bR\cap \bE$ such that $\text{reddeg}(F)=4$. Then such stable sheaves are following type:
%\begin{enumerate}
%\item $F$ is supported on a singular elliptic quartic curve contained in a quadric surface and $p \in \sing(C)$;
%\item $F$ is supported on the union of the twisted cubic curve with a secant line on the twisted cubic curve;
%\item $F$ is supported on the union of two non-coplanar conics.
%\end{enumerate}
%\end{proposition}

\noindent
We will next examine the case when the support of the sheaf is the union of a double line and a conic.
% (the case of Proposition \ref{prop:einf} (1)(c) and (2)(b).).
The picture becomes more complicated.
Case (1) of Proposition \ref{prop:reddeg3} below was already dealt with at Corollary \ref{non-reduced_degenerations}(1),
but we treat it also using modifications of sheaves in order to better illustrate the argument.

%The possible $C$ and $p$ are as follows.
%\begin{enumerate}
%\item $C=L_0^2 Q$ for a planar double line $L_0^2$ and a conic $Q$ where $p\in L_0$.
%\item $C=L_{-2}^2Q$ for a genus $-2$ double line $L_{-2}^2$ and a conic $Q$, where $p\in Q$ and $\dim T_p C=3$.
%\end{enumerate}
%
%Let $C=L\cup Q$ and $p\in L\cap Q$. Let us recall that our tautological family of stable sheaves parameterized
%$$
%\PP(\Ext^1(\O_{C\cup L}(p), \O_L(-1))\cong \PP^3.
%$$
%By the functoriality of Simpson space, there exists an injective map
%$$
%i:\PP^3 \lr \bE\subset \M_{\PP^3}(4m+1).
% $$
%Furthermore, one can easily check that the Kodaira-Spencer map of $i$ is an injective linear map and hence $i$ is a closed embedding.

\begin{proposition}
\label{prop:reddeg3}
In each of the following cases, all stable sheaves $F \in \PP(\Ext^1(\CC_p, \O_C))$ belong to $\bR \cap \bE$:
\begin{enumerate}
\item $C=L^2 Q$ is a union of a double line on $L$ and a conic $Q$ with $L$ and $Q$ meeting at a point, and $p\in L$;
\item $C=L^2Q$ is a union of a genus $-2$ double line $L^2$ and a conic $Q$, and $p\in Q\cap L$.
\end{enumerate}
\end{proposition}

\begin{proof}
%
%Item (2): Consider now the case when $C = L \cup C_0$.
%According to \cite{freiermuth_trautmann}, the closure of the set $\XX_0$ of structure sheaves  of twisted cubic curves
%is an irreducible component $\XX$ of $\M_{\PP^3}(3m+1)$.
%Moreover, $\XX \setminus \XX_0$ consists precisely of the sheaves of the form $\O_{C_0}(p)$,
%where $C_0$ is a planar singular cubic curve and $p \in \sing(C_0)$.
%Consider the open subset
%\[
%\mathbf{S} = \{ (L, G)\in \Hilb_{\PP^3}(m+1) \times \XX \mid \ L \cap \operatorname{supp}(\G) = \{ p' \},\, p' \neq p \}.
%\]
%There is an injective morphism $\mathbf{S} \to \M_{\PP^3}(4m+1)$ sending $(L, [\G])$ to the unique non-split extension of $\G$ by $\O_L(-1)$.
%Its image is irreducible and contains $\O_C(p)$.
%The generic member of the image of $\mathbf{S}$ is the structure sheaf of the union of a twisted cubic curve with an incident line,
%which lies in $\RR$. We deduce that $\O_C(p)$ can be approximated by elements in $\RR$.

(1) One can see that the unique extension sheaf $F\in\Ext^1(\CC_p, \O_C)=\CC$ fits into the following exact sequence
\begin{equation}\label{eqstable1}
\ses{\O_{L}(-1)}{F}{\O_{{L}\cup Q}}
\end{equation}
by using the isomorphism $\Ext^1(\CC_p, \O_{L}(-2)) \isom \Ext^1(\CC_p, \O_C)$
which is given by the exact sequence $\ses{\O_{L}(-2)}{\O_C}{\O_{L\cup Q}}$.
Also by direct computation, stable sheaves in \eqref{eqstable1} are parametrized by $\Ext^1(\O_{{L}\cup Q}, \O_{L}(-1))=\CC^3$.
We prove that the sheaves in \eqref{eqstable1} arise as elementary modifications of sheaves by using finite maps.
Let $f \colon C \to \PP^3$ be a map whose domain $C$ is the pair of two lines $L_1\cup L_2$
such that $f_{|L_1}$ is a degree two map onto $L$ and $f_{|L_2}$ is a bijection onto $Q$.
Then obviously $f$ can be regarded as an element in the moduli space $F_0(\PP^3,4)$.
Let $\Delta \subset F_0(\PP^3,4)$ be the locus of the finite maps whose image is the union of lines and smooth conics meeting at a point.
Then, the direct image sheaf $f_*\O_C$ fits into the exact sequence
\begin{equation}
\label{direct}\ses{\O_{{L}\cup Q}}{f_*\O_C}{\O_{L}(-1)}.
\end{equation}
We show by modifications of sheaves that all sheaves in \eqref{eqstable1} lie in $\bR$.
For example, see \cite{CK11}. To apply the modification technique, we need to choose a smooth chart of $F_0(\PP^3,4)$ at $f$.
Around $f$, by \cite[Theorem 0.1]{Parker}, the maps space $F_0(\PP^3,4)$ can be obtained as the $SL(2)$-quotient
%\[
%F_0(\PP^3,4) \isom F_0(\PP^1\times \PP^3, (1,4))/\Aut (\PP^1)
%\]
of the moduli space $F_0(\PP^1\times \PP^3, (1,4))$ of finite maps in $\PP^1\times \PP^3$ of bidegree $(1,4)$ where $\Aut (\PP^1)=SL(2)$ canonically acts on $F_0(\PP^1\times \PP^3, (1,4))$.
Among the fiber $f$ along the GIT-quotient map, let us choose that the graph map $f'$ whose restriction on $L_1$ is of bidegree $(1,2)$ which doubly covers $\PP^1\times L \subset \PP^1\times \PP^3$, then $f'$ has only the trivial automorphism.
%Without loss of generality, one can choose the $L_1$ is the main component of $f$ as the graph map.
Hence, around $f$, the space $F_0(\PP^1\times \PP^3, (1,4))$ is a smooth chart which is compatible with $SL(2)$-action. This implies that the argument in \cite[Lemma 4.6]{CK11} about the construction of the Kodaria-Spencer map of the maps space can be naturally applied in our setting.
Now, let us compute the normal space by the same technique as in \cite{CK11}.
We write $f$ as a composition $C \xrightarrow{h} C' \xrightarrow{g} \PP^3$, where $C'$ is a pair of lines.
Here $h$ is a two-to-one covering from $L_1$ and a bijection from $L_2$.
Also, $g$ is an isomorphism of degree $3$. By the octahedron axiom for the derived category, we have the exact sequence
\[
0 \to \Def(h) \to \Def(f) \to \Ext^1(h^*[g^*\Omega_{\PP^3} \to \Omega_{C'}], \O_C) \to \Ob(h)\isom \Ext^1(\Omega_C,\O_C) \to 0.
\]
Since the higher direct images $R^{\bullet}h_*(-)$ vanish, we have an isomorphism
\[
\Ext^1(h^*[g^*\Omega_{\PP^3} \to \Omega_{C'}], \O_C) \isom \Ext^1([g^*\Omega_{\PP^3} \to \Omega_{C'}], h_*\O_C).
\]
But the later space is decomposed into
\begin{multline}
\label{def2}
0 \to \Ext^1([g^*\Omega_{\PP^3} \to \Omega_{C'}], \O_{C'}) \to \Ext^1([g^*\Omega_{\PP^3} \to \Omega_{C'}], h_*\O_C) \to \\
\Ext^1([g^*\Omega_{\PP^3} \to \Omega_{C'}],O_L(-1))
\end{multline}
by the exact sequence
\[
\ses{\O_{C'}}{h_*\O_C}{\O_L(-1)}.
\]
The first term in \eqref{def2} is the deformation space of $g$ in $\PP^3$.
By using the local isomorphism $F_0(\PP^3,3) \isom \Hilb_{\PP^3}(3m+1)$ around $g$ associating $g$ to $g_*\O_{C'}=\O_{L\cup Q}$ (\cite[Theorem 1.4 and Proposition 3.3 (1)]{CK11}),
one can easily see that there exists a natural surjective homomorphism
\[
\psi \colon \Ext^1([g^*\Omega_{\PP^3} \to \Omega_{C'}],\O_{C'})\isom \Hom(I_{L\cup Q}, \O_{L\cup Q}) \to \Ext^1(\Omega_{L\cup Q}, \O_{L\cup Q})\isom \CC,
\]
where the map comes from the short exact sequence
\[
\ses{I_{L\cup Q}/I_{L\cup Q}^2}{\Omega_{\PP^3}|_{L\cup Q}}{\Omega_{L\cup Q}}.
\]
The kernel of $\psi$ is the deformation space of the map $g$ while fixing the node $p$ of the curve $L\cup Q$. In conclusion, we have a commutative diagram:
\[
\xymatrix
{
& \operatorname{Ker}(\psi) \ar@{^{(}->}[r] \ar@{^{(}->}[d] & \Ext^1([g^*\Omega_{\PP^3} \to \Omega_{C'}], \O_{C'})\ar@{->>}[r] \ar@{^{(}->}[d] & \Ext^1(\Omega_{L\cup Q}, \O_{L\cup Q})\ar@{=}[d] \\
\Def(h) \ar@{^{(}->}[r] & \Def(f)\ar[r] & \Ext^1([g^*\Omega_{\PP^3} \to \Omega_{C'}],h_*\O_C)\ar@{->>}[r] \ar@{->>}[d] & \Ob(g) \\
&& \Ext^1([g^*\Omega_{\PP^3} \to \Omega_{C'}], O_L(-1)).
}
\]
Combining with the deformation space $\Def(h) \isom \CC^2$ (\cite[Lemma 4.10]{CK11}),
we see that the normal space of $\Delta$ in $F_0(\PP^3,4)$ at $f$ is isomorphic to
\[
N_{\Delta/F_0(\PP^3,4),f} \isom \Ext^1([g^*\Omega_{\PP^3} \to \Omega_{C'}],O_L(-1)).
\]
Since the curve $L\cup Q$ is a locally complete intersection,
the two-terms complex $[g^*\Omega_{\PP^3} \to \Omega_{C'}]$ is quasi-isomorphic to $g^*N_{L\cup Q/\PP^3}^*[1]$. Thus,
\[
N_{\Delta/F_0(\PP^3,4),f}\isom \Ext^1(N_{L\cup Q/\PP^3}^*[1],\O_{L}(-1))=\Hom (N_{L\cup Q/\PP^3}^*, \O_{L}(-1)).
\]
Since the choice of the normal vectors in the tangent space of maps space at $f$ makes the switch of the sub/quotient sheaf of the original sheaf in \eqref{direct} (cf. \cite{CHK12}), the Kodaira-Spencer map
$
T_f F_0(\PP^3,4)\to \Ext^1(f_*\O_C,f_*\O_C)
$
descents to
\[
N_{\Delta/F_0(\PP^3,4),f}=\Hom (N_{L\cup Q/\PP^3}^*, \O_{L_0}(-1))\isom \Ext^1(\O_{L\cup Q},\O_{L_0}(-1))
\]
which is the compatible with the coboundary map given by the structure sequence
\[
\ses{I_{L\cup Q}}{\O}{\O_{L\cup Q}}.
\]
Therefore, the normal directions to $\Delta$ correspond exactly to the sheaves in the extension \eqref{eqstable1}.
This shows that such sheaves are obtained by elementary modifications, and hence lie in $\bR$.

\medskip

\noindent
(2) Similar arguments as above can be applied. Again, we want to show that all sheaves in the extension
\begin{equation}\label{eqstable2}
\ses{\O_{L}(-1)}{F}{\O_{{L}\cup Q}(p)}
\end{equation}
lie in $\bR$. Recall that in our convention $\O_{{L}\cup Q}(p)$ is the unique non-split extension of $\CC_p$ by $\O_{L \cup Q}$.
First, let us fix the point $p\in \PP^3$.
Let $\Delta_{1,3}$ be the locus of stable maps  $f \colon L_1\cup L_3 \to \PP^3$ such that degrees of $f$ on $L_1$ and $L_3$ are 1 and 3, respectively,
and $f$ maps $L_1\cap L_3$ to $p$.
Let $D$ be the locus in the boundary of $\Delta_{1,3}$ consisting of the finite maps $f:L_1\cup (L_1'\cup L_2) \to \PP^3$
such that the image $f|_{L_1'\cup L_2}$ is the union of a line $L$ and a coplanar conic and the restriction $f|_{L_1}$ is a bijection with the line $L$.
%where $L$ have the double covering of the reducible domain.
%Let us use the same notation by $\Delta_{1,2}$ and $D$ for the stable maps passing through $p$ at the node point. Without loss of the generality one can choose the map $f\in D$ as the general type.
Then the two spaces are smooth around $f$. %since they have the fiberation structure.

Let us describe the restricted Kodaira-Spencer map
\[
\xi \colon N_{D/\Delta_{1,3}, f} \to \Ext_{\PP^3}^1(\O_{L\cup Q}(p),\O_L(-1)).
\]
The normal space of $D$ in $\Delta_{1,3}$ consists of the two-dimensional contribution from smoothing two nodes of $L\cup Q$ and
the two-dimensional contribution from moving $L$ away from the plane containing $Q$.
By definition, $p$ is one node. We let $q$ be the other node.
From the exact sequence
\[
0 \to \O_{Q}(-q) \to \O_{L\cup Q}(p) \to \O_L \to 0,
\]
we have
\[
0 \to \Ext_{\PP^3}^1(\O_L,\O_L(-1)) \to \Ext_{\PP^3}^1(\O_{L\cup Q}(p),\O_L(-1)) \to \Ext_{\PP^3}^1(\O_Q(-q),\O_L(-1)) \to 0.
\]
The first term
\[
\Ext_{\PP^3}^1(\O_L,\O_L(-1)) \isom \H^0(N_{L/\PP^3}(-1)) \isom \CC^2
\]
is the deformation of $L$ while fixing the point $p$.
The third  term
\[
\Ext_{\PP^3}^1(\O_Q(-q),\O_L(-1)) \isom \Ext_L^1(\CC_p \oplus \CC_q, \O_L(-1)) \isom \CC \oplus \CC
\]
is the deformation space of smoothing the two nodes of the plane cubic curve $L\cup Q$ in the moduli space $\Hilb_{\PP^3}(3m+1)$ while fixing the point $p$.
After a diagram chase, we check that $\xi$ is an isomorphism and thus the modified sheaf lies in $\bR$.
Thus, we see that all sheaves in \eqref{eqstable2} belong to $\bR$.
\end{proof}

\medskip

Finally, we show that $\bR\cap \bE$ is irreducible. In fact, $\bR\cap \bE$ is the closure of the locus of sheaves from Proposition \ref{prop:reddeg4}(1)(a). Since it is straightforward that all other loci of $\bR\cap \bE$ is in the closure, we will show the sheaves from Propostion \ref{prop:reddeg3}(2) do not form a new irreducible component. We denote this locus by $\bE_{2b}$ following the notation in the previous section. We remark that $\bE_{2b}$ is irreducible.

\begin{lemma}\label{lem:dime2b}
If $F \in \bE_{2b}$ is generic, then $\ext^1(\F, \F) \le 19$.
\end{lemma}

\begin{proof}
%It remains to prove (iv).
Since $\bE_{2b}$ is irreducible, it is enough to prove the estimate for a single sheaf. Let $F$ be given by the resolution
\be
\label{resolution_1}
0 \to 3\O(-3) \xrightarrow{\psi} 5\O(-2) \xrightarrow{\f} \O(-1) \oplus \O \to \F \to 0
\ee
\[
\psi = \left[
\ba{ccc}
 -y &-z& 0  \\
 x & 0&  0 \\
0 & x&  0  \\
 0  &-y& x  \\
 0 & 0 & -y \\
\ea
\right], \ \
\f = \left[
\ba{ccccc}
x & y & z & 0&0 \\
0 & 0 &y^2 & xy & x^2
\ea
\right]
\]
This resolution is one of three possible resolution types we will encounter in Section \ref{sec:res}.
The support of $F$ is the subscheme defined by the ideal $\langle x^2,xy,y^3\rangle$ and the point $p$ is defined by $\langle x,y,z\rangle$.
Thus,  $F$ lies in $\bE_{2b}$. It is easy to check, with the help of the computer program Macaulay2 (\cite{M2}), that $\Ext^1(F,F) \cong \CC^{19}$.
\end{proof}

\begin{theorem}\label{theorem_18}
The variety $\bR\cap \bE$ is irreducible.
\end{theorem}

\begin{proof}
Suppose the variety $\bE_{2b} \subset \bR\cap \bE$ forms a new irreducible component. Choose a regular point $y$ of $\bE_{2b}$. Then
\begin{equation}
\label{equation_31}
\dim T_y\M_{\PP^3}(4m+1)  \ge \dim T_y \bR  + \dim T_y \bE - \dim T_y \bE_{2b} \ge 16+17-13 =20.
\end{equation}
Since we may choose $y$ to be generic, this contradicts the estimate at Lemma \ref{lem:dime2b}.
\end{proof}

\section{The resolutions of the sheaves in $\M(4m+1)$}
\label{sec:res}

We fix a $4$-dimensional vector space $V$ over $\CC$ and a basis $\{ X, Y, Z, W \}$ of $V^*$.
We identify $\PP^3$ with $\PP(V)$.
Let $\F$ be a one-dimensional sheaf on $\PP^3$.
The relevant part of the $E^1$-level of the Beilinson spectral sequence (\cite[Theorem 3.1.4]{OCS80}) converging to $\F$ is displayed in the following tableau:
\begin{align*}
& \H^1(\F(-1)) \tensor \O(-3) \stackrel{\f_1}{\lra} \H^1(\F \tensor \Omega^2(2)) \tensor \O(-2) \stackrel{\f_2}{\lra}
\H^1(\F \tensor \Omega^1(1)) \tensor \O(-1) \stackrel{\f_3}{\lra} \H^1(\F) \tensor \O \\
& \H^0(\F(-1)) \tensor \O(-3) \stackrel{\f_4}{\lra} \H^0(\F \tensor \Omega^2(2)) \tensor \O(-2) \stackrel{\f_5}{\lra}
\H^0(\F \tensor \Omega^1(1)) \tensor \O(-1) \stackrel{\f_6}{\lra} \H^0(\F) \tensor \O
\end{align*}
The $E^2$-level has tableau
\[
\xymatrix
{
\Ker(\f_1) \ar[rrd]^-{\f_7} & \Ker(\f_2)/\Image(\f_1) \ar[rrd]^-{\f_8} & \Ker(\f_3)/\Image(\f_2) & \Coker(\f_3) \\
\Ker(\f_4) & \Ker(\f_5)/\Image(\f_4) & \Ker(\f_6)/\Image(\f_5) & \Coker(\f_6)
}
\]
The spectral sequence degenerates at $E^3$, where all maps are zero:
\[
\xymatrix
{
\Ker(\f_7) & \Ker(\f_8) & \Ker(\f_3)/\Image(\f_2) & \Coker(\f_3) \\
\Ker(\f_4) & \Ker(\f_5)/\Image(\f_4) & \Coker(\f_7) & \Coker(\f_8)
}
\]
Thus $\f_7$ is an isomorphism, $\f_3$ is surjective, $\f_4$ is injective, $\Ker(\f_5) = \Image(\f_4)$ and
we have the exact sequence
\[
0 \to \Ker(\f_2)/\Image(\f_1) \xrightarrow{\f_8} \Coker(\f_6) \to \F \to \Ker(\f_3)/\Image(\f_2) \to 0
\]
Denote $p = \h^0(\F \tensor \Omega^1(1))$, $q = \h^0(\F \tensor \Omega^2(2))$.
Assume that $\F$ is semistable and has Hilbert polynomial $P_{\F}(m) = 4m+1$.
According to \cite{choi_chung_estimates} we have the relations
\[
\h^0(\F(-1)) = 0, \qquad \h^1(\F) = 0 \text{  or  } 1.
\]
The Beilinson monad with middle cohomology $\F$ yields an exact sequence
\be
\label{beilinson}
0 \to 3\O(-3) \oplus q \O(-2) \xrightarrow{\psi} (q+5)\O(-2) \oplus p\O(-1) \xrightarrow{\f} \Ker(\f_3) \oplus \H^0(\F) \tensor \O \to \F \to 0
\ee
in which $\psi_{12} = 0$ and $\f_{12} = 0$.
We recall two well-known facts.
The sheaves giving points in $\M_{\PP^3}(m+1)$ are precisely the structure sheaves of lines.
The sheaves giving points in $\M_{\PP^3}(2m+1)$ are precisely the structure sheaves of conic curves.

\begin{theorem}
Let $\F$ give a point in $\M_{\PP^3}(4m+1)$.
Then precisely one of the following is true:
\begin{enumerate}
\item[(i)]
$\h^0(\F \tensor \Omega^2(2)) = 0$, $\h^0(\F \tensor \Omega^1(1)) = 0$, $\h^0(\F) = 1$;
\item[(ii)]
$\h^0(\F \tensor \Omega^2(2)) = 0$, $\h^0(\F \tensor \Omega^1(1)) = 1$, $\h^0(\F) = 1$;
\item[(iii)]
$\h^0(\F \tensor \Omega^2(2)) = 1$, $\h^0(\F \tensor \Omega^1(1)) = 3$, $\h^0(\F) = 2$.
\end{enumerate}

\medskip

\noindent
The sheaves satisfying conditions (i) are precisely the sheaves having a resolution of the form
\be
\label{resolution_1}
0 \to 3\O(-3) \xrightarrow{\psi} 5\O(-2) \xrightarrow{\f} \O(-1) \oplus \O \to \F \to 0
\ee
\[
\f = \left[
\ba{ccccc}
l_1 & l_2 & l_3 & l_4 & l_5 \\
q_1 & q_2 & q_3 & q_4 & q_5
\ea
\right]
\]
where $\dim (\operatorname{span} \{ l_1, l_2, l_3, l_4, l_5 \}) \ge 3$.

\medskip

\noindent
The sheaves satisfying conditions (ii) are precisely the sheaves having a resolution of the form
\be
\label{resolution_2}
0 \to 3\O(-3) \xrightarrow{\psi} 5\O(-2) \oplus \O(-1) \xrightarrow{\f} 2\O(-1) \oplus \O \to \F \to 0
\ee
where $\f_{12} = 0$ and $\f_{11} \colon 5\O(-2) \to 2\O(-1)$ is not equivalent to a morphism
of the form
\[
\left[
\ba{ccccc}
\star & \star & 0 & 0 & 0 \\
\star & \star & \star & \star & \star
\ea
\right] \qquad \text{or} \qquad \left[
\ba{ccccc}
\star & \star & \star & \star & 0 \\
\star & \star & \star & \star & 0
\ea
\right]
\]

\medskip

\noindent
The sheaves satisfying conditions (iii) are precisely the sheaves having a resolution of the form
\be
\label{resolution_3}
0 \to \O(-4) \oplus \O(-2) \xrightarrow{\psi} \O(-3) \oplus 3\O(-1) \xrightarrow{\f} 2\O \to \F \to 0
\ee
\[
\psi = \left[
\ba{ll}
\phantom{-} l & \phantom{-} 0 \\
\phantom{-} 0 & \phantom{-} l \\
-f_1 & -l_1 \\
-f_2 & -l_2
\ea
\right], \qquad \f = \left[
\ba{cccc}
f_1 & l_1 & l & 0 \\
f_2 & l_2 & 0 & l
\ea
\right]
\]
where $l, l_1, l_2$ are linearly independent one-forms.
If $H \subset \PP^3$ is the plane given by the equation $l = 0$,
then $\F$ has resolution
\be
\label{resolution_4}
0 \to \O_H(-3) \oplus \O_H(-1) \xrightarrow{\bar{\f}} 2\O_H \to \F \to 0
\ee
\[
\bar{\f} = \left[
\ba{cc}
\bar{f}_1 & \bar{l}_1 \\
\bar{f}_2 & \bar{l}_2
\ea
\right]
\]
where $\bar{f}_1$, $\bar{f}_2$, $\bar{l}_1$, $\bar{l}_2$ denote classes modulo $l$.
\end{theorem}

\begin{proof}
(i) Assume first that $\h^0(\F) = 1$. The exact sequence (\ref{beilinson}) becomes
\[
0 \to 3\O(-3) \oplus q\O(-2) \xrightarrow{\psi} (q+5) \O(-2) \oplus p\O(-1) \xrightarrow{\f} (p+1)\O(-1) \oplus \O \to \F \to 0
\]
We claim that $p = 0$ or $1$. Indeed, if $p = 2$, then we would get a commutative diagram
\[
\xymatrix
{
p \O(-1) \ar[r]^-{\f_{22}} \ar[d] & \O \ar[r] \ar[d] & \O_L \ar[r] \ar[d] & 0 \\
(q+5) \O(-2) \oplus p\O(-1) \ar[r]^-{\f} & (p+1)\O(-1) \oplus \O \ar[r] & \F \ar[r] & 0
}
\]
Both $\O_L$ and $\F$ are stable and $\p(\O_L) = 1 > \p(\F)$, hence $\Hom(\O_L, \F) = 0$.
Thus $\O \to \F$ is the zero morphism. On the other hand $\H^0(\O) \to \H^0(\F)$ is injective because $\H^0(\Coker(\psi)) = 0$.
We have obtained a contradiction. If $p = 3$ or $p \ge 4$, then $\Coker(\f_{22})$ would be the structure sheaf of a point,
respectively, it would be zero. Both cases would yield contradictions as above.

Assume that $p=0$. Then $q=0$ because $\f_5$ is injective. We obtain the resolution
\[
0 \to 3\O(-3) \to 5\O(-2) \xrightarrow{\f} \O(-1) \oplus \O \to \F \to 0
\]
\[
\f = \left[
\ba{ccccc}
l_1 & l_2 & l_3 & l_4 & l_5 \\
q_1 & q_2 & q_3 & q_4 & q_5
\ea
\right]
\]
If $\dim (\operatorname{span} \{ l_1, l_2, l_3, l_4, l_5 \}) = 1$, then we may assume that $l_1 \neq 0$ and that
$l_2$, $l_3$, $l_4$, $l_5$ are zero.
We would get a commutative diagram
\[
\xymatrix
{
5\O(-2) \ar[r]^-{\f} \ar[d] & \O(-1) \oplus \O \ar[r] \ar[d] & \F \ar[r] \ar[d] & 0 \\
\O(-2) \ar[r]^-{l_1} & \O(-1) \ar[r] & \O_H(-1) \ar[r] & 0
}
\]
showing that $\F$ maps surjectively to $\O_H(-1)$. This is absurd, because $\dim (\operatorname{supp}(\F)) = 1$.
Likewise, if $\dim (\operatorname{span} \{ l_1, l_2, l_3, l_4, l_5 \}) = 2$, then $\F$ would have a quotient sheaf
of the form $\O_L(-1)$, in violation of semi-stability.
We conclude that $\F$ has resolution (\ref{resolution_1}).

Conversely, we assume that $\F$ has resolution (\ref{resolution_1}) and we must show that $\F$ is semistable.
At every point $P \in \PP^3$ we have $\operatorname{hd}_{P}(\F) \le 2$, hence $\operatorname{depth}_P(\F) \ge 1$.
From Grothendieck's Criterion we deduce that $\mathcal{H}^0_{\{P\}} (\F) = 0$, that is, $\F$ has no sections
supported on $\{ P \}$. Thus $\F$ has no zero-dimensional torsion. Assume that $\F$ had a destabilizing subsheaf $\E$.
We may assume that $\E$ is semistable.
Since $\h^0(\E) \le \h^0(\F) =1$ the Hilbert polynomial of $\E$ may be one of the following: $m+1$, $2m+1$, $3m+1$.
In the first case $\E \isom \O_L$ and its standard resolution fits into a commutative diagram
\[
\xymatrix
{
0 \ar[r] & \O(-2) \ar[r] \ar[d]^-{\gamma} & 2\O(-1) \ar[r] \ar[d]^-{\beta} & \O \ar[r] \ar[d]^-{\alpha} & \E \ar[r] \ar[d] & 0 \\
0 \ar[r] & 3\O(-3) \ar[r] & 5\O(-2) \ar[r] & \O(-1) \oplus \O \ar[r] & \F \ar[r] & 0
}
\]
Since $\alpha \neq 0$, we have $\Ker(\alpha) = 0$, hence $\Ker(\gamma) \isom \Ker(\beta) = 2\O(-1)$.
This is absurd.
In the second case $\E$ is the structure sheaf of a conic curve and its standard resolution fits into a commutative diagram
\[
\xymatrix
{
0 \ar[r] & \O(-3) \ar[r] \ar[d]^-{\gamma} & \O(-2) \oplus \O(-1) \ar[r] \ar[d]^-{\beta} & \O \ar[r] \ar[d]^-{\alpha} & \E \ar[r] \ar[d] & 0 \\
0 \ar[r] & 3\O(-3) \ar[r] & 5\O(-2) \ar[r] & \O(-1) \oplus \O \ar[r] & \F \ar[r] & 0
}
\]
Since $\alpha \neq 0$, we have $\Ker(\gamma) \isom \Ker(\beta)$. It follows that $\O(-1)$ is a subsheaf of $\Ker(\gamma)$.
This is absurd. Finally, assume that $P_{\E}(m) = 3m+1$. The quotient $\G = \F / \E$ has no zero-dimensional torsion
and $P_{\G}(m) = m$. It follows that $\G \isom \O_L(-1)$. We have a commutative diagram
\[
\xymatrix
{
0 \ar[r] & 3\O(-3) \ar[r] \ar[d]^-{\gamma} & 5\O(-2) \ar[r]^-{\f} \ar[d]^-{\beta} & \O(-1) \oplus \O \ar[r] \ar[d]^-{\alpha} & \F \ar[r] \ar[d] & 0 \\
0 \ar[r] & \O(-3) \ar[r] & 2\O(-2) \ar[r] & \O(-1) \ar[r] & \G \ar[r] & 0
}
\]
From the commutativity of the middle square we see that
\[
\f \sim \left[
\ba{ccccc}
\star & \star & 0 & 0 & 0 \\
\star & \star & \star & \star & \star
\ea
\right]
\]
This contradicts our hypothesis. We conclude that there are no destabilising subsheaves $\E \subset \F$.

\medskip

\noindent
(ii) We next examine the case when $\h^0(\F \tensor \Omega^1(1)) = 1$ and $\h^0(\F) = 1$.
Since $\f_5$ is injective we see that $q = 0$ or $1$. If $q=1$, then $\f_6$ would be generically zero,
hence $\f_6 = 0$ , hence $\Ker(\f_6)/\Image(\f_5) \isom \O_H(-1)$.
Recall that $\f_7 \colon \Ker(\f_1) \to \Ker(\f_6)/\Image(\f_5)$ is an isomorphism.
It would follow that $\O_H(-1)$ is a subsheaf of $3\O(-3)$. This is absurd. Thus $q=0$ and we have a resolution
\[
0 \to 3\O(-3) \to 5\O(-2) \oplus \O(-1) \xrightarrow{\f} 2\O(-1) \oplus \O \to \F \to 0
\]
If $\f_{11}$ were equivalent to a morphism of the form
\[
\left[
\ba{ccccc}
\star & \star & 0 & 0 & 0 \\
\star & \star & \star & \star & \star
\ea
\right]
\]
then $\F$ would have a quotient sheaf of the form $\O_H(-1)$, or of the form $\O_L(-1)$.
This, we saw above, yields a contradiction.
If $\f$ were equivalent to a morphism of the form
\[
\left[
\ba{cccccc}
\star & \star & \star & \star & 0 & 0 \\
\star & \star & \star & \star & 0 & 0 \\
\star & \star & \star & \star & q & l
\ea
\right]
\]
then we would have a commutative diagram
\[
\xymatrix
{
\O(-2) \oplus \O(-1) \ar[r]^-{[ q \ l]} \ar[d] & \O \ar[r] \ar[d] & \O_C \ar[r] \ar[d] & 0 \\
5\O(-2) \oplus \O(-1) \ar[r]^-{\f} & 2\O(-1) \oplus \O \ar[r] & \F \ar[r] & 0
}
\]
in which $C$ is the conic curve given by the equations $q = 0$, $l = 0$.
Both $\O_C$ and $\F$ are stable with $\p(\O_C) = 1/2 > \p(\F)$, hence $\Hom(\O_C, \F) = 0$.
Thus, the map $\O \to \F$ is zero. This, as we saw above, yields a contradiction.
We conclude that $\F$ has resolution (\ref{resolution_2}).

Conversely, if $\F$ has resolution (\ref{resolution_2}), then, by arguments analogous to the arguments in the case
of resolution (\ref{resolution_1}), we can show that $\F$ is semistable.

\medskip

\noindent
(iii) Finally, we consider the case when $\h^0(\F) = 2$. Then $p \ge 3$ and resolution (\ref{beilinson}) takes the form
\[
0 \to 3\O(-3) \oplus q \O(-2) \xrightarrow{\psi} (q+5) \O(-2) \oplus p \O(-1) \xrightarrow{\f} \Omega^1 \oplus (p-3) \O(-1) \oplus 2\O \to \F \to 0
\]
The morphism $\f_{32} \colon p\O(-1) \to 2\O$ cannot be equivalent to a morphism represented by a matrix of the form
\[
\left[
\ba{cccll}
\star & \cdots & \star & 0 & 0 \\
\star & \cdots & \star & l_1 & l_2
\ea
\right]
\]
otherwise we would have a commutative diagram
\[
\xymatrix
{
2\O(-1) \ar[r]^-{[ l_1 \ l_2 ]} \ar[d] & \O \ar[r] \ar[d] & \O_L \ar[r] \ar[d] & 0 \\
(q+5) \O(-2) \oplus p\O(-1) \ar[r] & \Omega^1 \oplus (p-3)\O(-1) \oplus 2\O \ar[r] & \F \ar[r] & 0
}
\]
But $\Hom(\O_L, \F) = 0$, hence the morphism $2\O \to \F$ is not injective.
On the other hand $\H^0(2\O) \to \H^0(\F)$ is injective because $\H^0(\Coker(\psi)) = 0$.
This yields a contradiction. Thus $p \le 5$. Denote $\E = \Coker(\f_{32})$.
Assume first that $p = 5$. We may write
\[
\f_{32} = \left[
\ba{ccccc}
X & Y & Z & W & 0 \\
0 & l_1 & l_2 & l_3 & l_4
\ea
\right]
\]
If $X$ and $l_4$ are linearly independent, then $\E$ is supported on the line $L$ given by the equations
$X = 0$, $l_4 = 0$.
Thus
\[
\E / \T^0(\E) \isom \O_L(d_1) \oplus \cdots \oplus \O_L(d_n)
\]
Since $\F$ is stable, $\Hom(\O_L(d), \F) = 0$ if $d \ge 0$.
Thus $\H^0(\E) \to \H^0(\F)$ is the zero morphism, hence $\H^0(2\O) \to \H^0(\F)$ is also the zero morphism.
This is a contradiction.
We have reduced to the case when
\[
\f_{32} = \left[
\ba{clllc}
X & Y & Z & W & 0 \\
0 & l_1 & l_2 & l_3 & X
\ea
\right]
\]
where $l_1$, $l_2$, $l_3$ are linearly independent one-forms in the variables $Y$, $Z$, $W$.
Note that
\[
\left[
\ba{lll}
Y & Z & W \\
l_1 & l_2 & l_3
\ea
\right] \nsim \left[
\ba{ccc}
0 & \star & \star \\
\star & \star & \star
\ea
\right]
\]
hence the maximal minors of this matrix
\[
q_1 = \left|
\ba{ll}
Z & W \\
l_2 & l_3
\ea
\right|, \qquad q_2 = \left|
\ba{ll}
Y & W \\
l_1 & l_3
\ea
\right|, \qquad q_3 = \left|
\ba{ll}
Y & Z \\
l_1 & l_2
\ea
\right|
\]
are linearly independent and have no common factor.
It follows easily that there is an exact sequence
\[
0 \to \O(-4) \xrightarrow{\beta} \O(-3) \oplus 3\O(-2) \xrightarrow{\alpha} 5\O(-1) \xrightarrow{\f_{32}} 2\O \to \E \to 0
\]
where
\[
\alpha = \left[
\ba{llll}
\phantom{-} 0 & - Y & - Z & - W \\
\phantom{-} q_1 & \phantom{-} X & \phantom{-} 0 & \phantom{-} 0 \\
- q_2 & \phantom{-} 0 & \phantom{-} X & \phantom{-} 0 \\
\phantom{-} q_3 & \phantom{-} 0 & \phantom{-} 0 & \phantom{-} X \\
\phantom{-} 0 & - l_1 & - l_2 & -l_3
\ea
\right], \qquad \beta = \left[
\ba{cccc}
-X & q_1 & - q_2 & q_3
\ea
\right]
\]
From this we get $P_{\E} = 3$, hence $\Hom(\E, \F) = 0$, hence $2\O_L \to \F$ is the zero morphism.
This, as we saw above, yields a contradiction.

\medskip

\noindent
Assume now that $p = 4$. We examine first the case when
\[
\f_{32} \sim \left[
\ba{clll}
X & Y & Z & 0 \\
0 & l_1 & l_2 & l_3
\ea
\right]
\]
If $X$ and $l_3$ are linearly independent, then $\E/\T^0(\E)$ is supported on a line and we get
a contradiction as above. Thus, we may write
\[
\f_{32} = \left[
\ba{cllc}
X & Y & Z & 0 \\
0 & l_1 & l_2 & X
\ea
\right]
\]
where $l_1$ and $l_2$ are linear forms in the variables $Y$, $Z$, $W$.
It is easy to see that there is an exact sequence
\[
0 \to 2\O(-2) \xrightarrow{\alpha} 4\O(-1) \xrightarrow{\f_{32}} 2\O \to \E \to 0
\]
where
\[
\alpha = \left[
\ba{ll}
- Y & - Z \\
\phantom{-} X & \phantom{-} 0 \\
\phantom{-} 0 & \phantom{-} X \\
- l_1 & - l_2
\ea
\right]
\]
We have $P_{\E}(m) = 2m+2$ and $\E$ has no zero-dimensional torsion.
From the semi-stability of $\F$ we see that the morphism $\E \to \F$ is zero or it factors
through a subsheaf $\F' \subset \F$ with $P_{\F'}(m) = m - k$, $k \ge 0$.
Thus $\H^0(\F') = 0$, so, at any rate, $\H^0(\E) \to \H^0(\F)$ is the zero map.
It follows that the map $\H^0(2\O) \to \H^0(\F)$ is zero, which yields a contradiction.

Assume next that
\[
\f_{32} = \left[
\ba{cccc}
l_{11} & l_{12} & l_{13} & l_{14} \\
l_{21} & l_{22} & l_{23} & l_{24}
\ea
\right] \nsim \left[
\ba{cccc}
\star & \star & \star & 0 \\
0 & \star & \star & \star
\ea
\right]
\]
Then we may assume that
\[
\f' = \left[
\ba{ccc}
l_{11} & l_{12} & l_{13} \\
l_{21} & l_{22} & l_{23}
\ea
\right] \nsim \left[
\ba{ccc}
0 & \star & \star \\
\star & \star & \star
\ea
\right]
\]
Note that the maximal minors of $\f'$ are linearly independent and have no common factor.
According to \cite{freiermuth_trautmann} and \cite{maican_duality}, the sheaf $\E' = \Coker(\f')$
gives a point in $\M_{\PP^3}(3m+2)$. Note that $\E$ is a quotient sheaf of $\E'$.
Since $\Hom(\E', \F) = 0$, it follows that $\Hom(\E, \F) = 0$, hence $2\O \to \F$ is the zero morphism.
We have reached again a contradiction.

\medskip

\noindent
Thus far we have proved that $p=3$.
Resolution (\ref{beilinson}) takes the form
\[
0 \to 3\O(-3) \oplus q \O(-2) \xrightarrow{\psi} (q+5) \O(-2) \oplus 3 \O(-1) \xrightarrow{\f} \Omega^1 \oplus 2\O \to \F \to 0
\]
The morphism $\f_{22} \colon 3\O(-1) \to 2\O$ has linearly independent maximal minors.
We claim that these maximal minors have a common linear factor.
If this were not the case, then, as mentioned above, $\Coker(\f_{22})$ would give a point in $\M_{\PP^3}(3m+2)$
and we would reach the contradictory conclusion that $2\O \to \F$ is the zero morphism.
It is clear now that $\Ker(\f_{22}) \isom \O(-2)$.
The isomorphism $\f_7 \colon \Ker(\f_1) \to \Ker(\f_{22})/\Image(\f_5)$ shows that $q = 1$.
Resolving $\Omega^1$ in the above sequence gives the resolution
\[
0 \to \O(-4) \oplus 3\O(-3) \oplus \O(-2) \xrightarrow{\psi} 4\O(-3) \oplus 6\O(-2) \oplus 3\O(-1) \xrightarrow{\f}
6\O(-2) \oplus 2\O \to \F \to 0
\]
\[
\psi = \left[
\ba{ccc}
\psi_{11} & \psi_{12} & 0 \\
0 & \psi_{22} & 0 \\
0 & \psi_{32} & \psi_{33}
\ea
\right], \qquad \f = \left[
\ba{ccc}
\f_{11} & \f_{12} & 0 \\
0 & \f_{22} & \f_{23}
\ea
\right]
\]
\[
\psi_{11} = \left[
\ba{c} X \\ Y \\ Z \\ W \ea
\right], \qquad \psi_{33} = \left[
\ba{l} \phantom{-} l_1 \\ -l_2 \\ \phantom{-} l_3 \ea
\right], \qquad \f_{23} = \left[
\ba{ccc}
l_{11} & l_{12} & l_{13} \\
l_{21} & l_{22} & l_{23}
\ea
\right]
\]
where
\[
u l_1 = \left|
\ba{cc}
l_{12} & l_{13} \\
l_{22} & l_{23}
\ea
\right|, \qquad u l_2 = \left|
\ba{cc}
l_{11} & l_{13} \\
l_{21} & l_{23}
\ea
\right|, \qquad u l_3 = \left|
\ba{cc}
l_{11} & l_{12} \\
l_{21} & l_{22}
\ea
\right|
\]
for some $u \in V^*$.
We claim that $\rank(\psi_{12}) = 3$.
To see this we dualise the above exact sequence.
According to \cite[Lemma 3]{maican_duality}, we have the exact sequence
\[
0 \to 2\O(-4) \oplus 6\O(-2) \xrightarrow{\f^\trans} 3\O(-3) \oplus 6\O(-2) \oplus 4\O(-1) \xrightarrow{\psi^\trans}
\O(-2) \oplus 3\O(-1) \oplus \O \to \F^\dual \to 0
\]
According to \cite{maican_duality}, $\F^\dual$ gives a point in $\M_{\PP^3}(4m-1)$.
If $\psi_{12}$ were zero, then we would get a commutative diagram
\[
\xymatrix
{
4\O(-1) \ar[r]^-{\psi_{11}^\trans} \ar[d] & \O \ar[r] \ar[d] & 0 \ar[d] \\
3\O(-3) \oplus 6\O(-2) \oplus 4\O(-1) \ar[r]^-{\psi^\trans} & \O(-2) \oplus 3\O(-1) \oplus \O \ar[r] & \F^\dual \ar[r] & 0
}
\]
showing that the morphism $\O \to \F^\dual$ is zero.
But the map $\H^0(\O) \to \H^0(\F^\dual)$ is injective because $\H^0(\Coker(\f^\trans)) = 0$.
If $\rank(\psi_{12}) = 1$, then the map $\O \to \F^\dual$ would factor through the structure sheaf of a point,
so it would be zero.
If $\rank(\psi_{12}) = 2$, then the map $\O \to \F^\dual$ would factor through the structure sheaf $\O_L$ of a line,
so it would be zero, because $\Hom(\O_L, \F^\dual) = 0$. We reach again contradictions.
This proves the claim.

Canceling $3\O(-3)$ we get the resolution
\[
0 \to \O(-4) \oplus \O(-2) \to \O(-3) \oplus 6\O(-2) \oplus 3\O(-1) \xrightarrow{\f} 6\O(-2) \oplus 2\O \to \F \to 0
\]
We have $\rank(\f_{12}) = 6$, otherwise $\F$ would map surjectively to $\O_H(-2)$ for a plane $H \subset \PP^3$.
This is clearly impossible. Canceling $6\O(-2)$ we finally get the resolution
\[
0 \to \O(-4) \oplus \O(-2) \xrightarrow{\psi} \O(-3) \oplus 3\O(-1) \xrightarrow{\f} 2\O \to \F \to 0
\]
with
\[
\psi = \left[
\ba{cl}
l & \phantom{-} 0 \\
f_1 & \phantom{-} l_1 \\
f_2 & -l_2 \\
f_3 & \phantom{-} l_3
\ea
\right], \qquad \f = \left[
\ba{cccc}
g_1 & l_{11} & l_{12} & l_{13} \\
g_2 & l_{21} & l_{22} & l_{23}
\ea
\right]
\]
The sheaf $\E = \Coker(\f_{12} \colon 3\O(-1) \to 2\O)$ is supported on $H \cup \{ P \}$, where $H$ is the plane given
by the equation $u=0$ and $P$ is the point given by the ideal $\langle l_1, l_2, l_3 \rangle$.
Since $\F$ is a quotient sheaf of $\E$ and since $\F$ has no zero-dimensional torsion,
we see that $\operatorname{supp}(\F) \subset H$.
Applying the snake lemma to the commutative diagram in which the middle row is the dual of the above exact sequence
\[
\xymatrix
{
& & 0 \ar[d] & 0 \ar[d] \\
& & \O(-1) \ar[d] \ar[r]^-{l} & \O \ar[r] \ar[d] & \O_{H'} \ar[r] & 0 \\
0 \ar[r] & 2\O(-4) \ar[r]^-{\f^\trans} & 3\O(-3) \oplus \O(-1) \ar[r]^-{\psi^\trans} \ar[d] & \O(-2) \oplus \O \ar[r] \ar[d] & \F^\dual \ar[r] & 0 \\
& & 3 \O(-3) \ar[d] \ar[r]^-{[ l_1 \ -l_2 \ l_3 ]} & \O(-2) \ar[r] \ar[d] & \CC_P \ar[r] & 0 \\
& & 0 & 0
}
\]
we get the exact sequence
\[
\O_{H'} \to \F^\dual \to \CC_P \to 0
\]
As $\F^\dual$ has no zero-dimensional torsion, we see that $P \in H'$, that is, $l \in \operatorname{span} \{ l_1, l_2, l_3 \}$.
Moreover, $\operatorname{supp}(\F^\dual) \subset H'$, hence also $\operatorname{supp}(\F) \subset H'$.
It follows that $H = H'$, otherwise $\F$ would be supported on a line,
yet a vector bundle of rank greater than one on $\PP^1$ is not stable.
Thus, we may assume that $u = l$ and we may write
\[
\psi = \left[
\ba{cl}
l & \phantom{-} 0 \\
f_1 & \phantom{-} l \\
f_2 & -l_2 \\
f_3 & -l_3
\ea
\right], \qquad \f =\left[
\ba{cccc}
g_1 & l_2 & l & 0 \\
g_2 & l_3 & 0 & l
\ea
\right]
\]
From the relations
\[
g_1 l + l_2 f_1 + l f_2 = 0, \qquad
g_2 l + l_3 f_1 + l f_3 = 0
\]
we see that $f_1$ is divisible by $l$.
Performing column operations on $\psi$, we may assume that $f_1 = 0$.
Thus $g_1 = -f_2$, $g_2 = -f_3$. We have obtained resolution (\ref{resolution_3}).

Conversely, given resolution (\ref{resolution_3}), we also have resolution (\ref{resolution_4}),
hence, by \cite[Theorem 3.2.1]{drezet_maican}, $\F$ is semistable.
\end{proof}

\begin{remark}
The general sheaves in $\bR$ and $\bE$ have the same resolution of the form \eqref{resolution_1}.
The general sheaves in $\bP$ and the sheaves in the wall-crossing have resolution \eqref{resolution_2}.
The stable sheaves $F$ with $\h^0(F)=2$ have resolution \eqref{resolution_3}.
\end{remark}

\end{document}